\documentclass[a4paper]{amsart}
\usepackage[utf8]{inputenc}
\usepackage{amsfonts}
\usepackage{amsmath,latexsym,amssymb,amsfonts, amsthm}
\usepackage{esint}
\usepackage{graphicx}
\usepackage{amscd}
\usepackage{color}
\usepackage{bm}             
\usepackage{enumerate}
\usepackage{aliascnt}
\usepackage[dvips]{epsfig}
\usepackage{psfrag}
\usepackage{enumitem}
\usepackage{epsfig}

\usepackage[
hmarginratio={1:1},     
vmarginratio={1:1},     
textwidth=16cm,        
textheight=21cm,
heightrounded,          
]{geometry}

\usepackage[
backend=biber,
maxnames=99,
style=alphabetic,
sorting=nyt,
doi=false,
url=false,
isbn=false
]{biblatex}
\AtEveryBibitem{%
	\clearlist{language}%
	\clearfield{langid}%
}
\defbibheading{mybib}[References]{%
	\section*{#1}%
	\markboth{}{}%
}
\usepackage{graphicx,color}
\usepackage[colorlinks]{hyperref}
\usepackage{cleveref}
\hypersetup{linkcolor=blue,citecolor=blue,filecolor=black,urlcolor=blue}
\usepackage{mathtools}

\renewcommand{\b}{\beta}

\newtheorem{theorem}{Theorem}[section]
\crefname{theorem}{theorem}{theorems}
\Crefname{theorem}{Theorem}{Theorems}

\newaliascnt{proposition}{theorem}
\newtheorem{proposition}[proposition]{Proposition}
\aliascntresetthe{proposition}
\crefname{proposition}{proposition}{propositions}
\Crefname{proposition}{Proposition}{Propositions}

\newaliascnt{corollary}{theorem}
\newtheorem{corollary}[corollary]{Corollary}
\aliascntresetthe{corollary}
\crefname{corollary}{corollary}{corollaries}
\Crefname{corollary}{Corollary}{Corollaries}

\newaliascnt{lemma}{theorem}
\newtheorem{lemma}[lemma]{Lemma}
\aliascntresetthe{lemma}
\crefname{lemma}{lemma}{lemmas}
\Crefname{lemma}{Lemma}{Lemmas}

\theoremstyle{definition}
\newaliascnt{remark}{theorem}
\newtheorem{remark}[remark]{Remark}
\aliascntresetthe{remark}
\theoremstyle{definition}
\crefname{remark}{remark}{remarks}
\Crefname{remark}{Remark}{Remarks}

\newaliascnt{example}{theorem}

\aliascntresetthe{example}
\theoremstyle{definition}
\crefname{example}{example}{examples}
\Crefname{example}{Example}{Examples}

\newaliascnt{definition}{theorem}

\aliascntresetthe{definition}
\theoremstyle{definition}
\crefname{definition}{definition}{definitions}
\Crefname{definition}{Definition}{Definitions}
	
\numberwithin{theorem}{section}

\numberwithin{equation}{section}

\newcommand{\R}{\mathbb{R}}
\newcommand{\N}{\mathbb{N}}

\newcommand{\mf}[1]{\mathbf{#1}}
\newcommand{\bs}[1]{\boldsymbol{#1}}

\newcommand{\pa}{\partial}
\renewcommand{\S}{\mathbb{S}}

\newcommand{\cH}{{\mathcal H}}

\newcommand{\cU}{{\mathcal U}}

\newcommand{\norm}[1]{\|#1\|}

\newcommand{\ssubset}{\subset\joinrel\subset}
\definecolor{my_red}{RGB}{255, 0, 0}

\definecolor{my_green}{RGB}{20, 200, 20}

\definecolor{my_purple}{RGB}{255, 76, 135}

\newcommand{\dist}{{\mathrm{dist}}}
\newcommand{\supp}{{\mathrm{supp}}}

\newcommand{\weak}{\rightharpoonup}

\newcommand{\eps}{\varepsilon}

\DeclareMathOperator{\loc}{loc}

\usepackage{accents}

\renewcommand{\epsilon}{\varepsilon}
\DeclareUnicodeCharacter{02BC}{'}

\author[L. Giaretto]{Lorenzo Giaretto}\thanks{}
\address{Lorenzo Giaretto\newline \indent
	Dipartimento di Matematica ``Giuseppe Peano'', Universit\`a di Torino, \newline \indent
	Via Carlo Alberto 10,
	10123 Torino, Italy}
\email{lorenzo.giaretto@unito.it}

\title[On $k$-wise interaction models: H\"older bounds and properties of the limits]{On elliptic systems with $k$-wise interactions in the strong competition regime: uniform H\"older bounds and properties of the limiting configurations}
\keywords{Competition-diffusion systems; $k$-wise interaction; partial segregation; singularly perturbed elliptic systems; H\"older bounds; Liouville-type theorems}
\subjclass[2020]{35R35; 35B25 (35J61; 35J47)}
\thanks{Declarations of interest: none. \\
	Data availability: Data sharing not applicable to this article as no datasets were generated or analysed during the current study.}

\begin{document}
	
	\maketitle
	
	\begin{abstract}
		In this paper we investigate a class of variational reaction–diffusion systems with strong competition driven by beyond-pairwise interactions. The model involves $d$ nonnegative components interacting through $k$-wise terms, with $3\leq k\leq d$, and includes symmetric interaction coefficients accounting for multi-component effects as well as suitable nonlinear terms. We focus on minimal energy solutions, proving uniform-in-$\beta$ H\"older bounds up to an explicit threshold exponent depending only on the dimension of the space and on the order $k$ of the interaction.
		
		As $\beta\to+\infty$, we show that minimizers converge strongly in $H^1$ and in H\"older spaces to a partially segregated configuration, characterized as minimizer of a natural variational problem under a $k$-segregation constraint. Finally, we prove that every minimizer of the limit problem enjoys the H\"older regularity and we derive some basic extremality conditions.
	\end{abstract}

\section{Introduction and main results}
Reaction-diffusion systems in the strong competition regime have been the object of extensive research in the last decades. The majority of the works are devoted to the case of binary interactions, where components interact pairwise, but recently phase separation driven by multiple interactions has appeared in the physical literature about multicomponent liquids and gases, motivating new mathematical research. We refer to \cite{Pet14,Po,CoCoOh,LQZ24} and the references therein for a physical explanation.

\medskip
The aim of this work is to study a variational model with a quite general \emph{beyond-pairwise} interaction, where the order can range from ternary up to the number of components of the system.
To be precise, fix two integers $3\leq k\leq d$, where $d$ is the number of components (densities) and $k$ is the order of the interaction.
Let $\Omega \subset \R^N, N\geq 2$ be a bounded smooth domain, and let $(\psi_1, \psi_2, \dots, \psi_d)$ be a $d$-tuple of nonnegative Lipschitz continuous functions in $\overline{\Omega}$ satisfying the \emph{partial $k$-segregation condition}, that is,
\begin{equation}\label{PSC}
\prod_{\ell=1}^k \psi_{i_\ell} \equiv 0 \quad \text{in $\overline{\Omega}$} \quad \text{for any subset of indexes } \{i_1,\dots,i_k\}\subseteq \{1,\dots,d\}. 
\end{equation}
In other words, at most $k-1$ functions can be positive at the same point $x_0$, hence their positivity sets can partially overlap. 
For better readability of the problem under study, we introduce the following notation:  we denote with $[d]:=\{1,2,\dots,d\}$ and with $|\cdot|$ the cardinality of a set. Then, for any $J\subseteq[d]$ with $|J|=k$, namely $J=\{j_1,\dots,j_k\}$, we write more compactly
\begin{equation*}
	u_J^2:=\prod_{\ell=1}^k u_{j_\ell}^2.
\end{equation*}
Our goal is to analyse the asymptotic behavior as $\beta \to +\infty$ of some solutions to the boundary value problem
\begin{equation}\label{eq:general problem better notation}
\begin{cases} 
	\displaystyle -\Delta u_i = -\beta u_i \sum_{\substack{J\subseteq [d]\setminus \{i\}\\ |J|=k-1}} \gamma_{J,i} u_J^2 + f_{i,\b}(x,u_{i}), \quad u_i\geq 0 & \text{in }\Omega,\\
	u_i = \psi_i & \text{on }\partial \Omega,
\end{cases}
\end{equation}
where $i=1,\dots,d$, $f_{i,\b}$ is a family of Caratheodory functions satisfying suitable growth restrictions (see \eqref{eq:growth assumption weak} and \eqref{eq:growth assumption strong} below) and $\gamma_{J,i}:=\gamma_{J\cup \{i\}}=\gamma_{i,j_1,\dots,j_{k-1}}>0$ are positive coefficients modelling the strength of the interaction among the components $i_1,\dots,i_k$, which we assume to be symmetric, namely
\begin{equation}\label{eq:symmetry interaction coefficients}
	\gamma_{J,i}=\gamma_{J',i'}=\gamma_L \text{ whenever } J\cup\{i\}=J'\cup\{i'\}=L.
\end{equation} 
Notice that the equation of each component involves all its possible $k$-wise interactions with the other components.

We begin by motivating the questions addressed in this paper. Assume that (as it turns out to be true for some choices of the nonlinearity that we will specify later) \eqref{eq:general problem better notation} possesses a gradient structure and that for every $\beta>0$ fixed a (weak) solution $\mf{u}_\beta$ of the problem can be found by minimizing the energy functional associated to the system in a Sobolev space of functions with fixed traces, and assume that such a family of minimizers $\{\mf u_\b\}_{\b>0}$ is bounded in $H^1(\Omega,\R^d)$ uniformly with respect to $\beta$. This implies that, up to a subsequence, $\mf{u}_\beta \rightharpoonup \tilde{\mf{u}}$ weakly in $H^1(\Omega,\R^d)$ for some $\tilde{\mf u}$, and the following questions arise quite naturally:
\begin{itemize}
	\item[($i$)] since $\mf{u}_\beta$ is $C^{1,\alpha}$ by standard elliptic regularity for each $\beta>0$ fixed (see Remark \ref{rmk: regularity beta fixed} below), can we improve the space in which uniform-in-$\beta$ bounds hold, passing from Sobolev to H\"older spaces?
	\item[($ii$)] what can be said about the limit $\tilde{\mf{u}}$?
\end{itemize} 
We will prove uniform-in-$\b$ H\"older bounds of such \emph{minimal energy} solutions for \emph{small} $\alpha$, and characterize the limit as a solution to a minimization problem featuring partial segregation.

\medskip

The rest of the section is organized as follows. In the first two subsections \ref{sub:stime holder} and \ref{sub:esistenza minimi e regolarità} we state the new results of the paper, then in subsections \ref{sub:novità} and \ref{sub:previous works} we highlight the elements of novelty and review some of the recent literature on partial segregation phenomena, respectively. Finally, in \ref{sub:further persp} we address possible further developments.

\subsection*{Basic notation}
We begin by introducing some basic notation that will be used throughout the paper.
\begin{itemize}
	\item We denote by $B_r(x_0)$ the ball of center $x_0$ and radius $r$ in $\R^N$, and by $S_r(x_0)=\partial B_r(x_0)$. We sometimes omit the center and simply write $B_r$ when $x_0$ is either clear from the context or irrelevant. 
	\item We use the vector notation $\mf{u}:=(u_1,\dots,u_d)$ for functions $\R^N \to \R^d$.
	\item As anticipated earlier, for any set of indexes $J\subseteq [d]=\{1,\dots,d\}$ and any exponent $\gamma$, we denote by
	$$
	u_J^\gamma:=\prod_{j\in J} u_{j}^\gamma.
	$$
	Furthermore, for every $J\subseteq [d]$ with $|J|=k$ we denote the coefficient of the corresponding interaction term as
	$$
	\gamma_J=\gamma_{I,i}=\gamma_{I',i'} \qquad \text{whenever $J=I\cup\{i\}=I'\cup\{i'\}$}.
	$$
	\item We say that a function $u$ is \emph{$\alpha$-H\"older continuous} in an open set $\Omega$ ($u\in C^{0,\alpha}(\Omega)$), for some $\alpha \in (0,1)$, if $u\in C(\overline{\Omega})$ and the H\"older seminorm 
	\[
	[u]_{C^{0,\alpha}(\Omega)}:= \sup_{\substack{x \neq y \\ x, y \in \Omega}} \frac{|u(x)-u(y)|}{|x-y|^\alpha}
	\]
	is finite. Notice that the supremum can be taken equivalently on $\overline{\Omega}$ and hence with this notation $C^{0,\alpha}(\Omega)=C^{0,\alpha}(\overline{\Omega})$. No limitation on the $L^\infty$ norm of $u$ is required.\\
	We say that a function $u$ is in $C^{0,\alpha}_{\loc}(\Omega)$ if it is in $C^{0,\alpha}(\Omega')$ for every open set $\Omega'\ssubset\Omega$.
\end{itemize}

\subsection{Uniform H\"older estimates with general nonlinearities}\label{sub:stime holder}

We begin by establishing some regularity results for minimizers of the energy functional associated with \eqref{eq:general problem better notation}, under rather mild assumptions on the nonlinear terms. Our main goal is to derive H\"older estimates that are uniform with respect to the competition parameter $\beta$, both in the interior of the domain and up to the boundary.

Let $\Omega\subset\R^N$ be a domain, not necessarily bounded, with $N\geq2$, and let $k,d$ be fixed integers as in the introduction. For every $i=1,\dots,d$ and $\beta>0$, let
\[
f_{i,\beta}:\Omega\times\R\to\R
\]
be a family of Carathéodory functions satisfying the following growth assumption:
\begin{equation}\label{eq:growth assumption weak}\tag{F}
	|f_{i,\beta}(x,s)|
	\leq C(\beta)(1+|s|^{p-1})
	\quad \text{for a.e. }x\in\Omega,\ \forall s\in\R,
	\qquad p\leq 2^*.
\end{equation}
where $C(\b)>0$ is independent of $i$.

We will later require the stronger, uniform-in-$\beta$ bound
\begin{equation}\label{eq:growth assumption strong}\tag{F'}
	|f_{i,\beta}(x,s)|
	\leq C(1+|s|^{p-1})
	\quad \text{for a.e. }x\in\Omega,\ \forall s\in\R,
	\qquad p\leq 2^*,
\end{equation}
where $C>0$ is independent of $\beta$ and $i$.

The two growth assumptions above play different roles in what follows. Assumption \eqref{eq:growth assumption weak} is mainly used to ensure that the potential terms in the energy functional are well defined and to obtain the interior regularity of solutions for each fixed $\beta$. In particular, its growth constant $C(\beta)$ does not enter quantitatively into the proof of the uniform H\"older estimate below. On the other hand, the uniform growth condition \eqref{eq:growth assumption strong}, whenever assumed, provides estimates on the nonlinear terms that are uniform with respect to $\beta$ and will play a more substantial role in the compactness analysis and in the derivation of uniform a priori bounds.

We then define
\[
F_{i,\beta}(x,u):=\int_0^u f_{i,\beta}(x,s)\,ds.
\]

Fix $\beta>0$, let $\bs\phi\in H^1(\Omega,\R^d)$ be a nonnegative function satisfying the partial segregation condition \eqref{PSC}, and let $\Omega'\ssubset\Omega$ be an arbitrary open subset. We consider the problem of minimizing the energy functional associated with system \eqref{eq:general problem better notation} in $\Omega'$ among functions with prescribed trace $\bs\phi$. More precisely, setting
\begin{equation}\label{eq:def J_beta}
	J_{\beta,\bs f_\beta}(\mf u,\Omega')
	:=
	\int_{\Omega'}
	\left(
	\frac12\sum_{i=1}^d|\nabla u_i|^2
	+\frac{\beta}{2}
	\sum_{\substack{J\subseteq[d]\\ |J|=k}}
	\gamma_J u_J^2
	-\sum_{i=1}^d F_{i,\beta}(x,u_i(x))
	\right)\,dx,
\end{equation}
and
\begin{equation}\label{eq: space with fixed traces}
	\mathcal U_{\bs\phi}
	:=
	\left\{
	\mf u\in H^1(\Omega',\R^d):
	\mf u-\bs\phi\in H_0^1(\Omega',\R^d)
	\right\},
\end{equation}
we consider
\begin{equation}\label{eq:inf J_beta}
	\inf_{\mf u\in\mathcal U_{\bs\phi}}
	J_{\beta,\bs f_\beta}(\mf u,\Omega').
\end{equation}
Notice that, due to the interaction term, $J_{\b,\bs f_{\b}}$ may fail to be finite on the whole class $\cU_{\bs\phi}$. This, however, does not affect the minimization problem, since $\bs\phi$ satisfies \eqref{PSC} and therefore provides an admissible competitor with finite energy. When no ambiguity arises, we will omit the dependence of the functional on the domain from the notation.

\begin{remark}\label{rmk: regularity beta fixed}
	If \eqref{eq:inf J_beta} admits a minimizer $\mf u_\b$ for a fixed $\b>0$, it is easy to see that $\mf u_\b$ is a weak solution of \eqref{eq:general problem better notation} in $\Omega'$ with boundary datum $\bs\phi$. In particular, it satisfies
	\[
	-\Delta u_{i,\b}
	\leq f_{i,\b}(x,u_{i,\b}(x))
	\quad \text{in }\Omega'.
	\]
	Since $\bs f_{\b}$ satisfies \eqref{eq:growth assumption weak}, a standard Brezis--Kato type argument (see \cite[Lemma B.3]{Struwe_2008}) yields
	\[
	u_{i,\b}\in L^q_{\loc}(\Omega')
	\qquad\text{for every }q\geq1.
	\]
	Calderón--Zygmund theory together with Sobolev embeddings then gives interior $C^{1,\alpha}$ regularity for $\mf u_\b$.
\end{remark}

The existence of such minimizers will be addressed in Subsection \ref{sub:esistenza minimi e regolarità}. For the moment, we take their existence as given and focus on estimates that are uniform as $\beta\to+\infty$.

\begin{theorem}[Interior uniform H\"older bounds and compactness]\label{thm: regolarità interna}
	Let $\Omega\subset\R^N$ be any open set, assume that $\{\bs f_{\b}\}_\b$ is a family of Carathéodory functions satisfying \eqref{eq:growth assumption weak}, and let
	\[
	\{\mf u_\beta=(u_{1,\beta},\dots,u_{d,\beta})\}_{\beta>0}
	\subset H^1_{\loc}(\Omega,\R^d)
	\]
	be a family of weak solutions to
	\[
	-\Delta u_{i,\b}
	=
	-\beta u_{i,\b}
	\sum_{\substack{J\subseteq[d]\setminus\{i\}\\ |J|=k-1}}
	\gamma_{J,i}u_J^2
	+f_{i,\b}(x,u_{i,\b}),
	\qquad
	u_{i,\b}\geq0
	\quad\text{in }\Omega,
	\]
	such that $\mf u_\beta$ is a minimizer of $J_{\b,\bs f_\b}$ with respect to variations with compact support, that is, for every $\Omega'\ssubset\Omega$,
	\begin{equation}\label{eq:local minimality}
		J_{\beta,\bs f_\b}(\mf u_{\beta},\Omega')
		\leq
		J_{\beta,\bs f_\b}(\mf u_\beta+\bs\varphi,\Omega')
		\qquad
		\forall\bs\varphi\in H_0^1(\Omega',\R^d),
	\end{equation}
	where $J_{\b,\bs f_\b}$ is defined in \eqref{eq:def J_beta}.
	
	Then:
	\begin{enumerate}[label=(\roman*), ref=(\roman*)]
		\item there exists $\bar\nu\in(0,1)$, defined in \eqref{eq:bar nu} and depending only on the dimension $N$ and the order of the interaction $k$, such that, for every $\Omega'\ssubset\Omega$ and every $\alpha\in(0,\bar\nu)$, there exists a constant $C=C(\alpha,N,k,d,\Omega,\Omega')>0$ such that
		\begin{equation}\label{eq:interior estimate thm}
			\|\mf u_\beta\|_{C^{0,\alpha}(\Omega')}
			\leq
			C\left(
			\|\mf u_\beta\|_{L^\infty(\Omega)}
			+
			\sup\left\{
			|f_{i,\b}(x,u)|:
			x\in\Omega,\ 
			u\in[0,\sup_\Omega u_{i,\b}],\
			i=1,\dots,d
			\right\}
			\right)
		\end{equation}
		for every $\b>0$;
		
		\item if, moreover, $\{\bs f_\b\}_\b$ satisfies \eqref{eq:growth assumption strong} and $\{\mf u_\beta\}_\beta$ is uniformly bounded in $L^\infty(\Omega)$, then there exists $\tilde{\mf u}\in H^1_{\loc}(\Omega,\R^d)$ such that, up to a subsequence,
		\[
		\mf u_\beta\to\tilde{\mf u}
		\quad\text{in }H^1_{\loc}(\Omega,\R^d)
		\text{ and in }C^{0,\alpha}_{\loc}(\Omega,\R^d)
		\quad\text{for every }\alpha\in(0,\bar\nu),
		\]
		and
		\[
		\beta\int_K
		\sum_{\substack{J\subseteq[d]\\ |J|=k}}
		\gamma_Ju_{J,\beta}^2\,dx
		\to0
		\qquad\text{for every compact }K\subset\Omega.
		\]
	\end{enumerate}
\end{theorem}

\begin{remark}
	We stress that the restriction on the H\"older exponent depends only on the order of the interaction $k$, and not on the number of components $d$.
\end{remark}

\medskip

We now establish the global counterpart of the previous theorem for minimizers with prescribed boundary data.

\begin{theorem}[Global uniform H\"older bounds for minimizers with fixed traces]\label{thm: regolarità al bordo}
	Let $\Omega\subset\R^N$ be a bounded domain with $C^1$ boundary, let $\boldsymbol\psi=(\psi_1,\dots,\psi_d)$ be a nonnegative Lipschitz continuous function on $\overline\Omega$ satisfying the partial segregation condition \eqref{PSC}, and let $\{\bs f_\b\}_\b$ be a family of Carathéodory functions satisfying \eqref{eq:growth assumption weak}. For every $\beta>0$, let
	\[
	\mf u_\beta=(u_{1,\beta},\dots,u_{d,\beta})
	\]
	be a nonnegative minimizer of
	\begin{equation}\label{eq:c_beta}
		c_{\beta,\bs f_\b}
		:=
		\inf_{\mf u\in\mathcal U_{\bs\psi}}
		J_{\beta,\bs f_\b}(\mf u,\Omega).
	\end{equation}
	Then:
	\begin{enumerate}[label=(\roman*), ref=(\roman*)]
		\item there exists $\bar\nu\in(0,1)$, defined in \eqref{eq:bar nu} and depending only on the dimension $N$ and the order of the interaction $k$, such that, for every $\alpha\in(0,\bar\nu)$,
		\begin{equation}\label{eq:global estimate}
			\|\mf u_\beta\|_{C^{0,\alpha}(\overline\Omega)}
			\leq
			C\left(
			\|\mf u_\beta\|_{L^\infty(\Omega)}
			+
			\sup\left\{
			|f_{i,\b}(x,u)|:
			x\in\Omega,\
			u\in[0,\sup_\Omega u_{i,\b}],\
			i=1,\dots,d
			\right\}
			\right),
		\end{equation}
		with $C=C(N,\alpha,k,d,\Omega,\bs\psi)>0$ independent of $\beta$;
		
		\item \label{thm point 2} if, in addition, $\{\bs f_\b\}_\b$ satisfies \eqref{eq:growth assumption strong} with $p<2^*$ and $\{\mf u_\b\}_\b$ is uniformly bounded in $H^1(\Omega,\R^d)$, then, for every $\alpha\in(0,\bar\nu)$, the family $\{\mf u_\b\}_\b$ is uniformly bounded in $C^{0,\alpha}(\overline\Omega,\R^d)$. Moreover, there exists $\tilde{\mf u}\in H^1(\Omega,\R^d)$, with
		\[
		\tilde{\mf u}-\bs\psi\in H_0^1(\Omega,\R^d),
		\]
		such that, up to a subsequence,
		\[
		\mf u_\beta\to\tilde{\mf u}
		\quad\text{in }H^1_{\loc}(\Omega,\R^d)
		\text{ and in }C^{0,\alpha}(\overline\Omega,\R^d)
		\quad\text{for every }\alpha\in(0,\bar\nu),
		\]
		and
		\[
		\beta\int_K
		\sum_{\substack{J\subseteq[d]\\ |J|=k}}
		\gamma_Ju_{J,\beta}^2\,dx
		\to0
		\qquad\text{for every compact }K\subset\Omega.
		\]
	\end{enumerate}
\end{theorem}

Some remarks concerning the first two results are in order.

\begin{remark}\label{rmk:bounds follow under A1-A2}
	The uniform $H^1$ boundedness assumed in Theorem \ref{thm: regolarità al bordo} \ref{thm point 2} follows automatically under additional growth restrictions on the forcing term. In particular, it is guaranteed under assumption (A2) stated in Subsection \ref{sub:esistenza minimi e regolarità}. Indeed, by minimality and \eqref{eq:growth assumption strong}, taking $\bs\psi$ as a competitor for $J_{\b,\bs f_\b}$, we obtain
	\[
	J_{\b,\bs f_\b}(\mf u_\b,\Omega)
	\leq J_{\b,\bs f_\b}(\bs\psi,\Omega)
	=
	\sum_{i=1}^d\int_\Omega
	\left(
	\frac12|\nabla\psi_i|^2
	-F_{i,\b}(x,\psi_i(x))
	\right)\,dx
	\leq C,
	\]
	for some $C>0$ depending on $\bs\psi$, but independent of $\b$. On the other hand, as we shall show in the proof of Proposition \ref{prop:minim at fixed trace exist},
	\[
	J_{\b,\bs f_\b}(\mf u_\b,\Omega)
	\geq
	\frac{\eps}{4}\|\mf u_\b\|_{H^1(\Omega)}^2
	-C'\|\mf u_\b\|_{H^1(\Omega)}
	-C'',
	\]
	for some constants $C',C''>0$ independent of $\b$. Combining the two inequalities yields the uniform $H^1(\Omega)$ boundedness of the family $\{\mf u_\b\}_{\b>0}$.
\end{remark}

\begin{remark}
	Without additional assumptions, it is not clear whether $\mf u_\beta\to\tilde{\mf u}$ in $H^1(\Omega)$, since one cannot control the behaviour of $\nabla\mf u_\beta$ near the boundary of $\Omega$. However, in the particular case of homogeneous Dirichlet boundary conditions, i.e. $\bs\psi=0$, one can prove $H^1(\Omega)$ convergence on the whole domain and the vanishing of the interaction integral on the whole $\Omega$ as $\beta\to+\infty$; see \cite{NoTaTeVe} for a similar argument.
\end{remark}

\begin{remark}
	Theorems \ref{thm: regolarità interna} and \ref{thm: regolarità al bordo} are quite general; however, they rely on the a priori existence of a family of minimizers for $J_{\b,\bs f_\b}$ in the class of functions with prescribed trace. As anticipated earlier, the existence of such minimizers is not guaranteed for an arbitrary family of Carathéodory functions with subcritical growth. We will address the issue of existence in Subsection \ref{sub:esistenza minimi e regolarità}.
	
	We also remark that, in the case $\bs f_\b=0$, minimizers exist by lower semicontinuity of the gradient norm and Fatou's lemma. In particular, when $k=d=3$, for a family of minimizers uniformly bounded in $L^\infty$, we recover \cite[Theorems 1.1 and 1.2]{SoTe1}.
\end{remark}

\begin{remark}
	The threshold $\bar\nu$ is defined in terms of an optimal ``overlapping partition problem'' on the sphere. More precisely,
	\begin{equation}\label{eq:bar nu}
		\bar\nu
		=
		\min_{\ell=2,\dots,k}\frac{\alpha_{\ell,N}}{\ell},
	\end{equation}
	where $\alpha_{\ell,N}$ is defined in \eqref{eq:alpha ell,N}. In Section \ref{sec:preliminar}, we will recall that $\alpha_{\ell,N}\leq2$ for every $\ell\geq2$, so that Theorem \ref{thm: regolarità interna} provides uniform H\"older bounds for every exponent smaller than $\bar\nu\leq2/k$. We do not expect this regularity to be optimal. In fact, when $k=d=3$ in the homogeneous case, the optimal regularity was determined in \cite{SoTe2} and is $C^{0,3/4}$, which is better than the $C^{0,2/3}$ regularity given by the present theorem. We refer to Subsection \ref{sub:further persp} for a discussion of this aspect.
\end{remark}

Finally, assuming a suitable convergence of the nonlinear terms as $\b\to+\infty$, we can characterize the limit profile as a minimizer of a natural limit problem and, in particular, improve the local $H^1$ convergence given by Theorem \ref{thm: regolarità al bordo} to global $H^1$ convergence.

To this aim, let $\Omega\subset\R^N$ be a bounded domain with $C^1$ boundary, and let $\boldsymbol\psi=(\psi_1,\dots,\psi_d)$ be a nonnegative Lipschitz continuous function on $\overline\Omega$ satisfying the partial segregation condition \eqref{PSC}. We introduce the set of partially segregated functions
\[
\mathcal \cH_{\bs\psi}
:=
\left\{
\mf u\in\mathcal U_{\bs\psi}:
\prod_{\ell=1}^k u_{j_\ell}\equiv0
\text{ in }\Omega
\text{ for every }J=\{j_1,\dots,j_k\}\subseteq[d]
\right\},
\]
and, for every given Carathéodory function
$\bs f=(f_1,\dots,f_d)$, the functional
\[
J_{\infty,\bs f}(\mf u)
:=
\sum_{i=1}^d\int_\Omega
\left(
\frac12|\nabla u_i|^2
-F_i(x,u_i(x))
\right)\,dx,
\qquad
\mf u\in\mathcal \cH_{\bs\psi},
\]
where
\[
F_i(x,u):=\int_0^u f_i(x,s)\,ds.
\]
We set
\begin{equation}\label{eq:c infty}
	c_{\infty,\bs f}
	:=
	\inf_{\mf u\in\mathcal \cH_{\bs\psi}}
	J_{\infty,\bs f}(\mf u).
\end{equation}

\begin{theorem}[Characterization of the limit problem]\label{thm:limit problem}
	Under the same hypotheses of Theorem \ref{thm: regolarità al bordo} \ref{thm point 2}, let $\{\b_n\}_n$ be a sequence such that $\b_n\to+\infty$ as $n\to\infty$. Assume moreover that there exists a Carathéodory function
	\[
	\bs f_\infty=(f_{1,\infty},\dots,f_{d,\infty})
	\]
	such that
	\begin{equation}\label{eq:convergence pointwise f_n to f}
		f_{i,\b_n}(x,s)
		\to f_{i,\infty}(x,s)
		\quad
		\text{for a.e. }x\in\Omega,\ \forall s\in[0,+\infty),
		\qquad i=1,\dots,d.
	\end{equation}
	
	Then, up to a subsequence,
	\[
	\mf u_{\b_n}\to\tilde{\mf u}
	\quad\text{in }H^1(\Omega,\R^d)
	\text{ and in }C^{0,\alpha}(\overline\Omega,\R^d)
	\quad\text{for every }\alpha\in(0,\bar\nu),
	\]
	and
	\[
	\b_n\int_\Omega
	\sum_{\substack{J\subseteq[d]\\ |J|=k}}
	\gamma_Ju_{J,\b_n}^2\,dx
	\to0
	\qquad\text{as }n\to\infty.
	\]
	Furthermore, $\tilde{\mf u}$ is a minimizer of problem \eqref{eq:c infty} with $\bs f=\bs f_\infty$, namely
	\[
	J_{\infty,\bs f_\infty}(\tilde{\mf u})
	=
	c_{\infty,\bs f_\infty},
	\]
	and the corresponding energy levels converge:
	\[
	c_{\b_n,\bs f_{\b_n}}
	\to c_{\infty,\bs f_\infty}
	\qquad\text{as }n\to\infty.
	\]
\end{theorem}

\subsection{Existence of minimizers and applications to partial segregation problems}\label{sub:esistenza minimi e regolarità}

The results stated in the previous subsection rely on the existence of minimizers, an issue that has not yet been addressed. We now introduce a set of assumptions ensuring the solvability of the minimization problems considered above for a rather large class of nonlinearities. As a further consequence, we will later derive some basic regularity properties of minimizers of the limit problem under the partial segregation constraint.

Recall that, given a Carathéodory function and a bounded domain $\Omega$,
\[
f:\Omega\times\R\to\R
\]
satisfying the growth condition
\begin{equation}\label{eq:growth condition single function}
	|f(x,s)|
	\leq C(1+|s|^{p-1})
	\quad\text{for a.e. }x\in\Omega,\ \forall s\in\R,
	\qquad C>0,\quad p\leq2^*,
\end{equation}
we define
\[
F(x,u):=\int_0^u f(x,s)\,ds.
\]
We introduce the following assumptions:
\begin{itemize}
	\item[(A1)] 
	$F(x,u)\geq F(x,-u)
	\quad\text{for a.e. }x\in\Omega,\ \forall u\geq0.$

	\item[(A2)] There exist $b\in L^\infty(\Omega)$ and $C>0$ such that
	\[
	F(x,u)
	\leq \frac{b(x)}{2}u^2+Cu
	\quad\text{for a.e. }x\in\Omega,\ \forall u\in[0,+\infty),
	\]
	and there exists $\eps>0$ such that
	\[
	\int_\Omega
	\left(
	|\nabla v|^2-b(x)v^2
	\right)\,dx
	\geq
	\eps\int_\Omega|\nabla v|^2\,dx
	\qquad
	\forall v\in H_0^1(\Omega).
	\]
\end{itemize}

We say that a family $\{\bs f_\b\}_{\b>0}$ satisfies {\rm (A1)} and {\rm (A2)} uniformly if, for every $i=1,\dots,d$ and every $\b>0$, the corresponding primitive $F_{i,\b}$ satisfies {\rm (A1)} and {\rm (A2)}, with the function $b$, the constant $C$, and the coercivity constant $\eps$ independent of both $i$ and $\b$.

\begin{proposition}\label{prop:minim at fixed trace exist}
	Let $\b>0$. If $f_{i,\b}$ satisfies {\rm (A1)} and {\rm (A2)} for every $i=1,\dots,d$, then the minimum in \eqref{eq:c_beta} is achieved by at least one nonnegative minimizer.
	
	Similarly, if $f_i$ satisfies {\rm (A1)} and {\rm (A2)} for every $i=1,\dots,d$, then the limit problem \eqref{eq:c infty} admits at least one nonnegative minimizer.
\end{proposition}

Since conditions {\rm (A1)} and {\rm (A2)} may appear somewhat restrictive at first sight, we provide some examples and observations.

\begin{enumerate}[label=(\roman*), ref=(\roman*)]
	\item Suppose that the family of primitives $\{\bs F_\b\}_\b$ is uniformly bounded from above. Since adding to each $F_{i,\b}$ a constant independent of $u$ only changes the corresponding energy by an additive constant, without affecting its minimizers, one may normalize the primitives so that
	\[
	F_{i,\b}(x,u)\leq0.
	\]
	After this normalization, condition {\rm (A2)} is immediately satisfied with $b\equiv 0$ and $C=0$. This observation will be useful in the sequel.
	
	An Allen--Cahn-type potential also fits naturally into this framework. Since we are interested in nonnegative solutions, a meaningful choice is a double-well potential with two nonnegative wells, for instance at $u=0$ and $u=1$. In our notation, one may take
	\[
	F_i(u)=-u^2(1-u)^2,
	\]
	which also satisfies \eqref{eq:growth condition single function} for $N=2,3,4$ (notice that, when $N=4$, $F_i$ has critical growth; nevertheless, the strict subcriticality assumption required in some of the results can be bypassed in this case, as explained in Remark \ref{rmk:subcritical not necessary with uniform L^infty}).
	
	Recall that the functional in \eqref{eq:def J_beta} contains the term $-F_{i,\b}$, so that the corresponding contribution to the energy is the usual nonnegative double-well potential
	\[
	u^2(1-u)^2.
	\]
	
	\item Let $\{\bs f_\b\}_\b$ be a family such that, for every $i$ and $\b$, the map $s\mapsto f_{i,\b}(x,s)$ is odd for a.e. $x\in\Omega$ and Lipschitz continuous, with Lipschitz constants uniformly bounded with respect to $i$ and $\b$. Assume moreover that there exist $\bar s>0$ and $\lambda<\lambda_1(\Omega)$, independent of $i$ and $\b$, such that
	\[
	f_{i,\b}(x,s)\leq\lambda s
	\quad\text{for a.e. }x\in\Omega,\ 
	\forall s\geq\bar s,\ 
	\forall\b>0,\ 
	i=1,\dots,d.
	\]
	Then the family satisfies both {\rm (A1)} and {\rm (A2)}.
	
	\item More generally, oddness in the $s$ variable immediately implies {\rm (A1)}, since in this case the corresponding primitive is even. In particular, whenever a family $\{f_{i,\b}\}_\b$ is initially defined only for nonnegative values of $s$, satisfies $f_{i,\b}(x,0)=0$, and its primitives satisfy {\rm (A2)}, one may extend each $f_{i,\b}$ to the negative half-line by oddness. The resulting primitive is even in $u$, and therefore {\rm (A1)} is automatically satisfied.
\end{enumerate}

\medskip

We now combine the previous results to establish regularity properties for minimizers of the partial segregation problem. By Theorem \ref{thm:limit problem}, a minimizer $\tilde{\mf u}$ of \eqref{eq:c infty} obtained as a strong competition limit belongs to $C^{0,\alpha}(\overline\Omega)$ for every $\alpha\in(0,\bar\nu)$. It is therefore natural to ask whether the same regularity holds for an arbitrary minimizer of \eqref{eq:c infty}. The following result gives a positive answer under the assumptions introduced above.

\begin{corollary}\label{cor: regolarità di tutti i minimi}
	Let $\Omega\subset\R^N$ be a bounded domain with $C^1$ boundary and let $\boldsymbol\psi=(\psi_1,\dots,\psi_d)$ be a nonnegative Lipschitz continuous function on $\overline\Omega$ satisfying the partial segregation condition \eqref{PSC}. Let
	\[
	\bs f=(f_1,\dots,f_d)
	\]
	be a family of Carathéodory functions satisfying \eqref{eq:growth condition single function} with $p<2^*$, as well as {\rm (A1)} and {\rm (A2)}, and let $\tilde{\mf u}$ be a minimizer of \eqref{eq:c infty}.
	
	Then
	\[
	\tilde{\mf u}\in H^1(\Omega,\R^d)
	\cap C^{0,\alpha}(\overline\Omega,\R^d)
	\qquad
	\text{for every }\alpha\in(0,\bar\nu).
	\]
	Moreover, for every $i=1,\dots,d$,
	\begin{equation}\label{eq:equaz in insieme positività}
		-\Delta\tilde u_i
		=
		f_i(x,\tilde u_i)
		\quad\text{in }\{\tilde u_i>0\}.
	\end{equation}
	
	If, in addition, for every $i=1,\dots,d$, $F_i$ is weakly differentiable
	with respect to $x$, $\nabla_x F_i$ is a Carathéodory function, and, for
	every $\Omega'\ssubset\Omega$ and $M>0$, there exists
	$h_{M,\Omega'}\in L^1(\Omega')$ such that
	\[
	|\nabla_x F_i(x,s)|
	\leq h_{M,\Omega'}(x)
	\qquad
	\text{for a.e. }x\in\Omega',\quad |s|\leq M,
	\]
	then $\tilde{\mf u}$ satisfies the following \emph{local Pohozaev identity}:
	\begin{multline}\label{eq:pohozaev}
		\int_{\partial B_r(x_0)}
		\sum_{i=1}^d|\nabla\tilde u_i|^2\,d\sigma
		=
		\frac{N-2}{r}
		\int_{B_r(x_0)}
		\sum_{i=1}^d|\nabla\tilde u_i|^2\,dx
		\\
		-\frac{2}{r}
		\int_{B_r(x_0)}
		\sum_{i=1}^d
		\left(
		NF_i(x,\tilde u_i)
		+\nabla_xF_i(x,\tilde u_i)\cdot(x-x_0)
		\right)\,dx
		\\
		+2
		\int_{\partial B_r(x_0)}
		\sum_{i=1}^d
		(\partial_\nu\tilde u_i)^2\,d\sigma
		+
		2
		\int_{\partial B_r(x_0)}
		\sum_{i=1}^d
		F_i(x,\tilde u_i)\,d\sigma,
	\end{multline}
	for every $B_r(x_0)\ssubset\Omega$.
\end{corollary}

\begin{remark}
	In the special case $k=d=3$ and $\bs f=\bs0$, the regularity statement above was already mentioned in \cite[Remark 1.8]{SoTe1}; we refer to the next subsection for a more detailed comparison with that work. A complete proof was omitted therein in order to avoid some technical complications, whereas the tools developed here allow us to establish the result directly.
	
	In particular, for future reference, we point out that in the homogeneous case the corollary shows that, for every $3\leq k\leq d$, any minimizer of
	$$
	\min\left\{\sum_{i=1}^d \int_\Omega |\nabla u_i|^2 : \mf u\in H^1(\Omega,\R^d), \mf u-\bs \psi\in H_0^1(\Omega,\R^d), \prod_{\ell=1}^k u_{j_\ell}\equiv 0 \ \text{in $\Omega$ for any $J\subseteq [d]$, $|J|=k$}\right\}
	$$
	belongs to $C^{0,\alpha}(\overline\Omega,\R^d)$ for every $\alpha\in(0,\bar\nu)$.
	
	We also mention that the previous regularity result admits a local counterpart for a larger class of profiles, namely limits of locally minimizing configurations for the Dirichlet energy with a competitive interaction term. We refer to the class $\mathcal L(\Omega)$ introduced in \cite{SoTe1}, whose definition can be naturally extended to the present setting. For the sake of brevity, we restrict here to minimizers of \eqref{eq:c infty}; a more general treatment will be developed in a forthcoming work.
\end{remark}

\begin{remark}\label{rmk:subcritical not necessary with uniform L^infty}
	In the previous results, the strict subcriticality assumption $p<2^*$ is only used to derive
	uniform $L^\infty$ bounds from uniform $H^1$ bounds. Consequently,
	Theorem \ref{thm: regolarità al bordo}~\ref{thm point 2} remains valid
	under the critical growth condition $p\leq 2^*$, provided that the family
	$\{\mf u_\beta\}_\beta$ can be shown to be uniformly bounded in
	$L^\infty(\Omega)$ by other means. The same observation applies to
	Corollary \ref{cor: regolarità di tutti i minimi}: its conclusion remains
	valid in the critical case, provided that the approximating minimizers
	constructed in its proof can be shown to be uniformly bounded in
	$L^\infty(\Omega)$.
	
	As a simple example, consider the Allen--Cahn-type potential
	\[
	F_i(s)=-s^2(1-s)^2,
	\]
	whose associated contribution to the energy is the nonnegative
	double-well potential $-F_i(s)=s^2(1-s)^2$, with wells at $0$ and $1$.
	If the boundary data satisfy
	\[
	0\leq\psi_i\leq1,
	\qquad i=1,\dots,d,
	\]
	then the truncation
	\[
	T(s):=\min\{1,\max\{0,s\}\}
	\]
	applied componentwise does not increase the energy: it decreases the
	Dirichlet term and the double-well potential, and it does not increase
	the nonnegative interaction term. Therefore, the minimizers of the associated functional with the $\beta$-interaction term can be
	chosen to satisfy
	\[
	0\leq u_{i,\beta}\leq1
	\qquad\text{a.e. in }\Omega, \forall \beta>0
	\]
	In this case, the required uniform $L^\infty$
	bounds follow directly from the variational structure, independently
	of any strict subcriticality assumption.
\end{remark}

\subsection{Main novelties}\label{sub:novità}

The present work extends the uniform H\"older estimates obtained in \cite{SoTe1} in several directions, by considering systems with an arbitrary number of components and arbitrary orders of multiple interaction, while also allowing for the presence of nonlinear terms. More precisely, \cite{SoTe1} deals with the homogeneous case of \eqref{eq:general problem better notation}, with $d=k=3$ and $\bs f_\b=\bs 0$, and establishes uniform a priori H\"older bounds for every exponent $\alpha\in(0,\bar\nu)$, where $\bar\nu=\alpha_{3,N}/3\leq 2/3$ and $\alpha_{3,N}$ is defined in \eqref{eq:alpha ell,N}. The proofs of Theorems \ref{thm: regolarità interna} and \ref{thm: regolarità al bordo} are inspired by the aforementioned work and are likewise based on a contradiction argument involving a blow-up analysis. A further generalization is that, unlike in \cite{SoTe1}, where the family of solutions is assumed to be uniformly bounded in $L^\infty$, the first parts of Theorems \ref{thm: regolarità interna} and \ref{thm: regolarità al bordo} provide explicit estimates of the $C^{0,\alpha}$ norm in terms of the $L^\infty$ norm and of the nonlinear forcing terms, without any a priori uniform boundedness assumption on the family.

The presence of the nonlinear forcing terms introduces additional technical difficulties that are absent in the homogeneous setting. First, some care is required already at the variational level, since suitable assumptions on the nonlinearities are needed to ensure that the associated energy functionals admit minimizers. Moreover, in order to apply the uniform H\"older estimates to such minimizers, one needs a priori bounds that are uniform with respect to the competition parameter $\beta$. While uniform $H^1$ bounds can be obtained under appropriate coercivity assumptions, deriving uniform $L^\infty$ bounds requires additional arguments. To this end, we develop an iterative procedure combining Brezis--Kato-type estimates with a Stampacchia-type truncation argument.

The higher-order nature of the interaction introduces further substantial difficulties in the blow-up analysis. In particular, as both the number of components $d$ and the order of the interaction $k$ increase, the number of possible blow-up scenarios grows rapidly, making it no longer feasible to enumerate all possible configurations and treat them separately. Moreover, new Liouville-type theorems are needed to deal with the different limiting problems arising in the blow-up analysis. A key ingredient in overcoming these difficulties is an inductive strategy that allows us to avoid an exhaustive classification of all possible blow-up configurations. More precisely, we partition the components into subsets according to their common asymptotic behavior and, relying on the new rigidity results established in this paper, derive a contradiction for every possible cardinality of these subsets.

\subsection{Previous works on partial segregation phenomena}\label{sub:previous works}

We first recall that the case of binary interactions, namely $d\geq k=2$, which leads in the limit to total segregation phenomena, has been extensively studied over the last decades; see, among others, \cite{Conti_Terracini_Verzini_2005,CaLi08,NoTaTeVe}. In this setting, uniform a priori bounds can be established in $C^{0,\alpha}$ for every $\alpha\in(0,1)$, while the optimal regularity of the limiting configurations is Lipschitz continuity. We also refer to \cite{CoTeVe,CaLi07,TaTe,Dancer_Wang_Zhang_2012,SoTaTeZi,Quitalo_2013} for finer results concerning this problem. Similar questions have been investigated in more general settings: we refer to \cite{SoaTer22} for a two-phase anisotropic problem, to \cite{Wei_Weth_2008} for a broader class of interaction terms, to \cite{TerVerZil14,TerVerZil15} for systems driven by the fractional Laplacian, and to \cite{Caffarelli_Patrizi_Quitalo_2017,Soave_Tavares_Terracini_Zilio_2018} for systems with long-range interactions.

\medskip

To the best of our knowledge, \cite{SoTe1} and its companion paper \cite{SoTe2} are among the first works to investigate partial segregation phenomena in a variational setting. In particular, \cite{SoTe2} shows that the optimal regularity for problem \eqref{eq:c infty} in the case $k=d=3$ is $C^{0,3/4}$. We also refer to \cite{BoBuFo,CaRo}, where related questions are addressed for systems with symmetric interactions and the optimal regularity turns out to be Lipschitz continuity. More recently, \cite{Bozorgnia_Arakelyan_2025} studied the singular limit of ternary interactions from the point of view of $\Gamma$-convergence, while \cite{Bozorgnia_2026} compares the variational and symmetric models, showing that, although the corresponding limiting configurations are substantially different, their free boundaries coincide.

\medskip

We also mention \cite{Giaretto_Soave_2026}, where a related problem involving a $k$-wise interaction term is studied. The authors establish the existence and qualitative properties of ground states and least-energy fully non-trivial solutions, namely solutions whose components are all nonzero. In the competitive regime, they obtain uniform $H^1$ bounds for solutions that are not minimizing in the sense considered in the present work, but rather minimize the energy within the class of fully non-trivial radial solutions.

\subsection{Further perspectives}\label{sub:further persp}

The uniform H\"older bounds established in the present work, together with the basic properties of the limiting configurations, represent a first crucial step towards the study of singularly perturbed systems with higher-order interactions, the regularity of minimizers of \eqref{eq:c infty}, and the structure of their free boundaries. We conclude this introduction by mentioning two open problems that deserve further investigation.

\medskip

A first question concerns the optimality of the H\"older threshold $\bar\nu$. In the case $k=d=3$ and $\bs f=\bs 0$, the limit problem reduces to the study of harmonic maps satisfying a ternary partial segregation condition, which was addressed in \cite{SoTe2}. There it is proved that minimizers belong to $C^{0,3/4}$ and that uniform H\"older bounds hold for every $\alpha\in(0,3/4)$, thus improving the a priori bounds for exponents up to $2/3$ established in \cite{SoTe1}. Somewhat surprisingly, the proof of this improved regularity relies on the previously established a priori H\"older bounds, and a direct proof of the optimal estimate is still not available.

In view of this, one may hope to develop a similar strategy in the more general setting considered here and obtain higher, possibly optimal, regularity by exploiting the a priori H\"older bounds established in the present work. The optimal exponent is expected to depend on the order of the interaction $k$ and possibly also on the number of components $d$. Moreover, a detailed description of the free boundary and of the nodal set of minimizers of the limit problem in this general framework is currently in preparation, in the spirit of the results obtained in \cite{TaTe} and \cite{SoTe2} for binary and ternary interactions, respectively.

\medskip

A second important question is whether our results can be extended to \emph{any} $L^\infty$-bounded family of solutions to \eqref{eq:general problem better notation}, without assuming that the solutions minimize the associated energy functional. This is known to hold for systems with binary interactions, since the aforementioned literature deals with general families of solutions. For higher-order interactions, however, the situation is considerably less clear.

Indeed, the blow-up argument developed in the present work relies on the minimality of the solutions at several key points; see Lemmas \ref{lemma: minimality implies positivity}, \ref{lemma: not 1 to k-2 divergent}, and \ref{lemma: not M unb e v bdd} below. At present, we are not able to remove this assumption. To the best of our knowledge, the only result available for families of non-minimal solutions in the context of beyond-pairwise interactions is contained in the aforementioned work \cite{Giaretto_Soave_2026}. There, the solutions are assumed to be minimizing only within the class of fully non-trivial radial functions, and radial symmetry is exploited to characterize the limiting profiles. Moreover, uniform H\"older bounds are not available in that setting, and only uniform $H^1$ bounds are obtained.

\subsection*{Structure of the paper}

In Section \ref{sec:preliminar}, we prove Proposition \ref{prop:minim at fixed trace exist} and collect some preliminary and known results needed throughout the paper. In Section \ref{sec:new liou}, we establish new Liouville-type theorems tailored to our framework. These tools are then used in Section \ref{sec: thm regolarità interna} to prove Theorem \ref{thm: regolarità interna}, and in Section \ref{sec: thm regolarità bordo} to prove Theorem \ref{thm: regolarità al bordo}. Finally, Section \ref{sec:limit problem} is devoted to the proof of Theorem \ref{thm:limit problem} and Corollary \ref{cor: regolarità di tutti i minimi}.

\section{Preliminaries and known results}\label{sec:preliminar}

The aim of this section is to establish some preliminary tools that will be useful in the sequel.

\subsection{Existence of minimizers with fixed trace}

We first prove that assumptions {\rm (A1)} and {\rm (A2)} guarantee the existence of minimizers for the problems \eqref{eq:c_beta} and \eqref{eq:c infty}. The argument is similar to the one used in \cite[Theorem 3.1]{Conti_Terracini_Verzini_2005}.

\begin{proof}[Proof of Proposition \ref{prop:minim at fixed trace exist}]
	We first consider the problem \eqref{eq:c_beta}. By {\rm (A2)}, for every $\mf u\in\mathcal U_{\bs\psi}$ we have
	\begin{align*}
		J_{\b,\bs f_\b}(\mf u,\Omega)
		&\geq
		\frac{1}{2}\sum_{i=1}^d
		\int_{\Omega}
		\left(
		|\nabla u_i|^2-b(x)u_i^2
		\right)\,dx
		-C\sum_{i=1}^d\int_{\Omega}|u_i|\,dx\\
		&\geq
		\frac{\eps}{4}
		\sum_{i=1}^d\int_{\Omega}|\nabla u_i|^2\,dx
		-C'\|\mf u\|_{H^1(\Omega)}\\
		&\geq
		\frac{\eps}{4}\|\mf u\|_{H^1(\Omega)}^2
		-C'\|\mf u\|_{H^1(\Omega)}
		-C'',
	\end{align*}
	where in the last line we used Poincaré's inequality on
	$\mf u-\bs\psi\in H_0^1(\Omega,\R^d)$, and $C',C''>0$ are constants depending on $\Omega$ and $\bs\psi$. This shows that $J_{\b,\bs f_\b}$ is coercive on $\mathcal U_{\bs\psi}$.
	
	Let $\{\mf u_n\}_n\subset\mathcal U_{\bs\psi}$ be a minimizing sequence. By {\rm (A1)} and the nonnegativity of $\bs\psi$, replacing
	\[
	\mf u_n=(u_{1,n},\dots,u_{d,n})
	\]
	with
	\[
	(|u_{1,n}|,\dots,|u_{d,n}|)
	\]
	does not increase the energy and preserves the prescribed trace. We may therefore assume that $\mf u_n$ is nonnegative.
	
	By coercivity, $\{\mf u_n\}_n$ is bounded in $H^1(\Omega,\R^d)$. Hence, up to a subsequence, there exists a nonnegative $\mf u\in\mathcal U_{\bs\psi}$ such that
	\[
	\mf u_n\rightharpoonup\mf u
	\quad\text{in }H^1(\Omega,\R^d),
	\qquad
	\mf u_n\to\mf u
	\quad\text{in }L^2(\Omega,\R^d)
	\]
	and
	\[
	\mf u_n(x)\to\mf u(x)
	\quad\text{for a.e. }x\in\Omega.
	\]
	Here we used the fact that $\Omega$ is bounded.
	
	We now show that the energy is lower semicontinuous along this sequence. By {\rm (A2)}, for every $i=1,\dots,d$,
	\[
	-F_{i,\b}(x,u_{i,n})
	+\frac{b(x)}{2}u_{i,n}^2
	+Cu_{i,n}
	\geq0.
	\]
	Since $u_{i,n}\to u_i$ almost everywhere, Fatou's lemma yields
	\begin{align*}
		\int_{\Omega}
		\left(
		-F_{i,\b}(x,u_i)
		+\frac{b(x)}{2}u_i^2
		+Cu_i
		\right)\,dx
		\leq
		\liminf_{n\to\infty}
		\int_{\Omega}
		\left(
		-F_{i,\b}(x,u_{i,n})
		+\frac{b(x)}{2}u_{i,n}^2
		+Cu_{i,n}
		\right)\,dx.
	\end{align*}
	On the other hand, the strong $L^2$ convergence, together with $b\in L^\infty(\Omega)$, gives
	\[
	\int_{\Omega}b(x)u_{i,n}^2\,dx
	\to
	\int_{\Omega}b(x)u_i^2\,dx,
	\qquad
	\int_{\Omega}u_{i,n}\,dx
	\to
	\int_{\Omega}u_i\,dx.
	\]
	Therefore,
	\[
	-\int_{\Omega}F_{i,\b}(x,u_i)\,dx
	\leq
	\liminf_{n\to\infty}
	\left(
	-\int_{\Omega}F_{i,\b}(x,u_{i,n})\,dx
	\right).
	\]
	
	Moreover, by the weak lower semicontinuity of the Dirichlet energy,
	\[
	\sum_{i=1}^d\int_{\Omega}|\nabla u_i|^2\,dx
	\leq
	\liminf_{n\to\infty}
	\sum_{i=1}^d\int_{\Omega}|\nabla u_{i,n}|^2\,dx,
	\]
	while, since the interaction term is nonnegative, Fatou's lemma gives
	\[
	\int_{\Omega}
	\sum_{\substack{J\subseteq[d]\\|J|=k}}
	\gamma_Ju_J^2\,dx
	\leq
	\liminf_{n\to\infty}
	\int_{\Omega}
	\sum_{\substack{J\subseteq[d]\\|J|=k}}
	\gamma_Ju_{J,n}^2\,dx.
	\]
	Combining these inequalities, we obtain
	\[
	J_{\b,\bs f_\b}(\mf u,\Omega)
	\leq
	\liminf_{n\to\infty}
	J_{\b,\bs f_\b}(\mf u_n,\Omega),
	\]
	and hence $\mf u$ is a nonnegative minimizer.
	
	The proof for the limit problem \eqref{eq:c infty} is analogous. Indeed, the same coercivity and lower semicontinuity arguments apply, with no interaction term, while the partial segregation constraint is preserved under almost everywhere convergence. Hence \eqref{eq:c infty} is also achieved by a nonnegative minimizer.
\end{proof}

\subsection{Definition of $\alpha_{\ell,N}$}
We briefly recall the definition of $\alpha_{\ell,N}$ appearing in the statement of the previous section, referring to \cite{SoTe1} for an exhaustive discussion of its relation with ACF-type monotonicity formulas.

For $u \in H^1(\S^{N-1})$, let
\[
\lambda(u) := \inf\left\{ \frac{\int_{\S^{N-1}}  |\nabla \varphi|^2 \,d\sigma   }{\int_{\S^{N-1}} \varphi^2 \,d\sigma} 
\left| \begin{array}{l} \varphi \in H^1(\S^{N-1}) \setminus \{0\} \ \text{and}\\ \mathcal{H}^{N-1}(\{\varphi \neq 0\} \cap \{u=0\}) = 0.  \end{array}\right.
\right\}.
\]
where $d\sigma = d\sigma_x$ and $\mathcal{H}^{N-1}$ is the $(N-1)$-dimensional Hausdorff measure and with the convention that $\lambda(u) = +\infty$ if $u \equiv 0$ on $\S^{N-1}$. In particular, if $u$ is continuous then $\lambda(u)$ is the first eigenvalue of the Laplace-Beltrami operator with homogeneous Dirichlet boundary conditions on the open set $\{u >0\}$.

Defining $\gamma: [0,+\infty) \to [0,+\infty)$ as
\[
\gamma(t) := \sqrt{\left(\frac{N-2}{2}\right)^2 + t\,} - \frac{N-2}{2},
\]
we consider the following optimization problem with partial segregation:
\begin{equation}\label{eq:alpha ell,N}
	\alpha_{\ell,N}:=\inf\left\{ \sum_{j=1}^\ell \gamma(\lambda(u_j)) \left| \begin{array}{l} u_j \in H^1(\S^{N-1})  \\  
		\int_{\S^{N-1}} u_1^2 \cdots u_\ell^2\, d\sigma = 0.
	\end{array}\right.\right\},
\end{equation}
The infimum is greater than or equal to $0$, since we are minimizing the sum of nonnegative quantities. We recall that, by the Friedland-Hayman inequality, $\alpha_{2,N}=2$, referring the interested reader to \cite[Chapter 2]{Petrosyan_Shahgholian_Uraltseva_2012} and the references therein. Moreover, we recall this estimate proved in \cite[Lemma 2.3]{SoTe1}. 

\begin{lemma}\label{lemma: stima alpha l,N}
	For every $\ell \geq 2$ we have $0 <\alpha_{\ell,N} \le 2$. 
\end{lemma}

\subsection{A decay estimate}
We now recall an estimate which turns out to be useful in proving exponential decay of suitable sequences. We refer to \cite[Lemma 3.1]{NoTaTeVe} for the proof.
\begin{lemma}\label{lemma:decadimento esponenziale}
	Let $B_R\subset \R^N$ be any ball of radius $R$. Let $M, A$ be positive constants, $h \in L^2(B_{2R})$, and let $u\in H^1(B_{2R})$ be a solution of
	$$
	\left\{\begin{array}{rlll}
		-\Delta u & \leq & -Mu + h & {\rm in }\ B_{2R}\\
		u & \geq & 0 & {\rm in}\ B_{2R}\\
		u & \leq & A & {\rm on}\ \partial B_{2R}.
	\end{array}\right.
	$$
	Then there holds
	$$\|u\|_{L^2(B_{R})}\leq CA e^{-\frac{R}{2} \sqrt{M}}+\frac{1}{M}\|h\|_{L^2(B_{2R})},$$
\end{lemma}

\subsection{A Liouville-type theorem}

We state here a Liouville-type theorem for globally H\"older continuous solutions of certain elliptic problems (see \cite[Proposition 3.2]{SoTe1} for a proof). Recall that, for some $\alpha \in (0,1)$, we say that $v$ is globally $\alpha$-H\"older continuous in $\R^N$ if the seminorm $[v]_{C^{0,\alpha}(\R^N)}$ is finite, imposing no limitation on the $L^\infty$ norm.

\begin{proposition}\label{thm: Liouville 1 comp}
	Let $\alpha \in (0,1)$, and let $w$ be a globally $\alpha$-H\"older continuous function in $\R^N$. Suppose moreover that one of the following equations is satisfied by $w$:
	\begin{itemize}
		\item[($i$)] either $\Delta w=0$ in $\R^N$;
		\item[($ii$)] or $\Delta w=\lambda$ in $\R^N$ for some $\lambda \in \R$;
		\item[($iii$)] or else $\Delta w = \lambda w$ with $\lambda >0$ in $\R^N$.  
	\end{itemize}
	Then $w$ must be constant and, in case ($iii$), $w \equiv 0$.
\end{proposition}

\section{New Liouville-type theorems}\label{sec:new liou}
\subsection{Liouville-type theorems for $k$-partially segregated functions}

The aim of this subsection is to establish new Liouville-type theorems for systems modelling partial segregation, generalizing the results obtained in \cite[Section 2]{SoTe1} for ternary interactions. In particular, we need to establish new rigidity results for both $k$-partially segregated functions and for solutions to elliptic systems with $k$-wise interactions. Although all the key tools are already available in the aforementioned reference, and the arguments presented here are essentially straightforward consequences of those results, our new formulation makes it possible to address a more general setting with $d$ components and $k$-wise interactions.

\medskip

First of all, let us recall a result for totally segregated components that will serve as the base step of an inductive argument (see \cite[Proposition 2.2]{NoTaTeVe} for the proof). 

\begin{proposition}\label{prop: Liou sys 2}
	Let $\alpha \in (0,1)$, and let $u,v \in H^1_{\loc}(\R^N) \cap C(\R^N)$ be globally $\alpha$-H\"older continuous functions in $\R^N$ such that
	$$
		\begin{cases}
			-\Delta u \le 0   \\
			-\Delta v \le 0 \\
			u \, v \equiv 0,
		\end{cases}
		\qquad u, v \ge 0, \qquad  \text{in $\R^N$}.
	$$
	Then one component of $(u,v)$ must vanish.
\end{proposition}

We now recall a fundamental Liouville-type theorem for partially segregated functions (see \cite[Theorem 3.4]{SoTe1}).

\begin{theorem}\label{thm: liouville partial segr implies one constant}
	Let $k \ge 3$ and $N \ge 2$ be positive integers, and let $\alpha \in (0,\alpha_{k,N}/k)$. For $i=1,\dots,k$, let $u_i \in H^1_{\loc}(\R^N) \cap C(\R^N)$ be globally $\alpha$-H\"older continuous functions in $\R^N$, such that
	\[
	-\Delta u_i \le 0 \quad \text{and} \quad u_i \ge 0 \quad \text{in $\R^N$},
	\]
	and moreover the partial segregation condition holds:
	\[
	\prod_{i=1}^k u_i  \equiv 0 \quad \text{in $\R^N$}.
	\]
	Then at least one function $u_i$ must be constant.
\end{theorem} 

We can actually prove this stronger result, which is an inductive generalization of \cite[Corollary 3.5]{SoTe1}

\begin{corollary}\label{cor: liou seg}
	Let $N \ge 2, k\geq 2$ be positive integers, and let $$\alpha \in (0,\min_{\ell=2,\dots,k} \alpha_{\ell,N}/\ell).$$
	For $i=1,\dots,k$, let $u_i \in H^1_{\loc}(\R^N) \cap C(\R^N)$ be globally $\alpha$-H\"older continuous functions in $\R^N$, such that
	\[
	-\Delta u_i \le 0 \quad \text{and} \quad u_i \ge 0 \quad \text{in $\R^N$},
	\]
	and moreover the partial segregation condition holds:
	\[
	\prod_{i=1}^k u_i \equiv 0 \quad \text{in $\R^N$}.
	\]
	Then at least one component $u_i$ vanishes identically.
\end{corollary}

\begin{proof}
	We prove the result by induction on $k\geq 2$.
	
	\medskip
	
	\emph{Base step: $k=2$}
	
	Since $\alpha_{2,N}=2$ by the Friedland-Hayman inequality, this is precisely the thesis of Proposition \ref{prop: Liou sys 2}.
	
	\medskip
	
	\emph{Inductive step}
	
	Assume the thesis is true for some $k\geq 2$, we aim at proving it for $k+1$. Assume then that
	$$\alpha \in (0,\min_{\ell=2,\dots,k+1} \alpha_{\ell,N}/\ell)$$
	and consider $k+1$ subharmonic nonnegative functions globally $\alpha$-H\"older continuous satisfying
	\begin{equation}\label{eq:7}
		\prod_{i=1}^{k+1} u_{i} \equiv 0.
	\end{equation}
	Then, by Theorem \ref{thm: liouville partial segr implies one constant}, at least one component $u_j$ is constant. If $u_j=0$, we are done. Otherwise, by \eqref{eq:7}, the other $k$ components satisfy 
	$$\prod_{\substack{i=1\\ i\neq j}}^{k+1} u_{i} \equiv 0.$$
	Then, by the inductive hypothesis, since in particular
	$$\alpha \in (0,\min_{\ell=2,\dots,k} \alpha_{\ell,N}/\ell),$$
	another component vanishes identically. This proves the inductive step.
\end{proof}

\subsection{Liouville theorem for solutions to elliptic systems with $k$-wise interaction} 
To reach a contradiction in the proof of Theorems \ref{thm: regolarità interna} and \ref{thm: regolarità al bordo}, we will make use of a new Liouville-type theorem for entire solutions to a system featuring $k$-wise interactions, which generalizes \cite[Corollary 1.4]{SoTe1} to our context.

\begin{corollary}[Liouville-type theorem]\label{cor: liouville intere interazione}
	Let $d\geq k\geq2$ and $N\geq2$, and, for $i=1,\dots,d$, let
	$u_i\in H^1_{\loc}(\R^N)\cap C(\R^N)$ be nonnegative functions satisfying
	\begin{equation}\label{eq:entire system liouville}
		\Delta u_i
		=
		u_i
		\sum_{\substack{J\subseteq[d]\setminus\{i\}\\ |J|=k-1}}
		M_{J,i}u_J^2
		\quad\text{in }\R^N,
		\qquad i=1,\dots,d,
	\end{equation}
	for some nonnegative coefficients
	\[
	\{M_L:L\subseteq[d],\ |L|=k\}
	\]
	satisfying the symmetry condition \eqref{eq:symmetry interaction coefficients}.
	Assume also that
	\[
	0\leq u_i(x)\leq C(1+|x|^\alpha)
	\qquad\text{for every }x\in\R^N,
	\]
	for some $\alpha<\alpha_{k,N}/k$. Then every component $u_i$ is constant.
	
	In particular, if $M_L>0$ for every $L\subseteq[d]$ with $|L|=k$,
	then at least $d-k+1$ components vanish identically, while all the
	remaining components are constant.
\end{corollary}

In order to prove Corollary \ref{cor: liouville intere interazione}, we recall
the following Liouville-type theorem, which follows from
\cite[Theorem 3.6]{SoTe1} with a suitable choice of the coupling term $g_i$.

\begin{theorem}\label{thm: liou sub}
	Let $k\geq2$ and $N\geq2$, and, for $i=1,\dots,k$, let
	$u_i\in H^1_{\loc}(\R^N)\cap C(\R^N)$ be nonnegative functions satisfying
	\[
	-\Delta u_i
	\leq
	-Mu_i\prod_{j\neq i}u_j^2
	\quad\text{in }\R^N,
	\]
	for some $M>0$.
	Assume moreover that
	\[
	0\leq u_i(x)\leq C(1+|x|^\alpha)
	\qquad\text{for every }x\in\R^N,
	\]
	with $\alpha<\alpha_{k,N}/k$. Then at least one component $u_i$
	vanishes identically.
\end{theorem}

\begin{proof}
	If one of the components vanishes identically, there is nothing to prove.
	Otherwise, by the strong maximum principle, all the components are positive.
	The Alt--Caffarelli--Friedman monotonicity formula in
	\cite[Theorem 2.6]{SoTe1}, from which \cite[Theorem 3.6]{SoTe1} follows,
	also holds for positive \emph{subsolutions} to
	\[
	-\Delta u_i
	+
	u_i^{q_i}g_i(x,\hat{\mf u}_i)
	\leq0,
	\qquad
	u_i>0
	\quad\text{in }\R^N,
	\quad i=1,\dots,k,
	\]
	where $q_i\geq1$ and
	$g_i\in C(\R^N\times[0,+\infty)^{k-1},[0,+\infty))$
	satisfies assumptions {\rm (H1)} and {\rm (H2)} therein.
	The case $k=2$ was first proved in \cite[Lemma 2.5]{NoTaTeVe}.
	
	Choosing
	\[
	q_i=1,
	\qquad
	g_i(x,\hat{\mf u}_i)
	=
	M\prod_{j\neq i}u_j^2,
	\]
	the assumptions are readily verified, and the conclusion follows by the
	same argument as in \cite[Theorem 3.6]{SoTe1}.
\end{proof}

\begin{proof}[Proof of Corollary \ref{cor: liouville intere interazione}]
	Fix any $L\subseteq[d]$ with $|L|=k$ and $M_L>0$. For every $j\in L$,
	equation \eqref{eq:entire system liouville} gives
	\[
	-\Delta u_j
	\leq
	-M_Lu_j
	\prod_{\substack{\ell\in L\\ \ell\neq j}}u_\ell^2
	\quad\text{in }\R^N.
	\]
	By Theorem \ref{thm: liou sub}, at least one component $u_j$, $j\in L$,
	vanishes identically. Hence
	\[
	\prod_{j\in L}u_j\equiv0
	\qquad
	\text{for every }L\subseteq[d],\ |L|=k,\ M_L>0.
	\]
	
	It follows that every interaction term on the right-hand side of
	\eqref{eq:entire system liouville} vanishes identically: indeed, terms
	corresponding to $M_L=0$ vanish trivially, while, if $M_L>0$, the
	previous conclusion gives
	\[
	u_i\prod_{\ell\in L\setminus\{i\}}u_\ell^2\equiv0
	\qquad\text{for every }i\in L.
	\]
	Therefore $\Delta u_i=0$ in $\R^N$ for every $i=1,\dots,d$.
	Since the components are nonnegative, the classical Liouville theorem
	implies that they are all constant.
	
	Finally, if $M_L>0$ for every $L\subseteq[d]$ with $|L|=k$, every
	$k$-tuple contains a component which vanishes identically. Hence at most
	$k-1$ components can be nonzero, and therefore at least $d-k+1$
	components vanish identically.
\end{proof}

\section{Interior uniform H\"older bounds}\label{sec: thm regolarità interna}

This section is devoted to the proof of Theorem \ref{thm: regolarità interna}. The main step consists in establishing a local estimate on balls, from which the general statement follows by a standard covering and scaling argument. The proof proceeds by contradiction and blow-up. Although some of the ideas originate from the analysis developed in \cite{SoTe1}, the presence of an arbitrary number of components, higher-order interactions, and nonlinear forcing terms requires a much more involved approach.

By translation and scaling, it is enough to work under the assumption that
\[
B_4\subset\Omega.
\]
Fix $\alpha\in(0,\bar\nu)$. We claim that there exists a constant
$C=C(N,\alpha,k,d)>0$ such that, for every $\beta>0$,
\begin{equation}\label{eq:Linf alpha estimate on balls}
	\|\mf u_\b\|_{C^{0,\alpha}(B_1)}
	\leq
	C\left(
	\|\mf u_\b\|_{L^\infty(B_3)}
	+
	\sup\left\{
	|f_{i,\b}(x,u)|:
	\begin{array}{l}
		x\in B_3,\ 
		u\in[0,\sup_{B_3}u_{i,\b}],\\
		i=1,\dots,d
	\end{array}
	\right\}
	\right).
\end{equation}
For convenience, throughout this section we set
\[
R_\b
:=
\sup\left\{
|f_{i,\b}(x,u)|:
x\in B_3,\ 
u\in[0,\sup_{B_3}u_{i,\b}],\
i=1,\dots,d
\right\}.
\]

Let $\eta\in C_c^1(\R^N)$ be radially decreasing, with
\[
\eta\equiv1\quad\text{in }\overline{B_1},
\qquad
\eta\equiv0\quad\text{in }B_3\setminus B_2.
\]
It is sufficient to prove
\begin{equation}\label{eq:localized holder estimate}
	[\eta\mf u_\b]_{C^{0,\alpha}(B_2)}
	\leq
	C\left(
	\|\eta\mf u_\b\|_{L^\infty(B_2)}+R_\b
	\right).
\end{equation}


Suppose then, by contradiction, that \eqref{eq:localized holder estimate} is false. There exist a sequence $\{\b_n\}_n\subset (0,+\infty)$, a family $\{\bs f_\b\}_{\b>0}$, and corresponding solutions
\[
\mf u_n:=\mf u_{\b_n},
\]
satisfying \eqref{eq:local minimality}, such that
\[
\frac{
	[\eta\mf u_n]_{C^{0,\alpha}(B_2)}
}{
	\|\eta\mf u_n\|_{L^\infty(B_2)}+R_n
}
>n,
\qquad n\in\N,
\]
where $R_n:=R_{\b_n}$. Set
\begin{equation}\label{eq:L_n}
	L_n
	:=
	[\eta\mf u_n]_{C^{0,\alpha}(B_2)}
	>
	n\left(
	\|\eta\mf u_n\|_{L^\infty(B_2)}+R_n
	\right).
\end{equation}

For each fixed $n$, the components of $\mf u_n$ are $C^{1,\alpha}_{\loc}$ by Remark \ref{rmk: regularity beta fixed}. After extracting a subsequence and relabelling the components, we may therefore find $x_n,y_n\in\overline{B_2}$, $x_n\neq y_n$, such that
\[
L_n
=
\frac{
	|(\eta u_{1,n})(x_n)-(\eta u_{1,n})(y_n)|
}{
	|x_n-y_n|^\alpha
}.
\]
We may also assume that $x_n\in B_2$: indeed, the two maximizing points cannot both belong to $\partial B_2$, since $\eta=0$ there.

Let $r_n:=|x_n-y_n|$. Then, by \eqref{eq:L_n},
\[
r_n^\alpha
=
\frac{
	|(\eta u_{1,n})(x_n)-(\eta u_{1,n})(y_n)|
}{L_n}
\leq
\frac{
	2\|\eta\mf u_n\|_{L^\infty(B_2)}
}{L_n}
<
\frac{2}{n}\to 0,
\]
as $n\to\infty$.

\medskip

We now zoom in around $x_n$ at the scale $r_n$. Two different rescalings will be used throughout the argument:
\[
v_{i,n}(x)
:=
\eta(x_n)
\frac{
	u_{i,n}(x_n+r_nx)
}{
	L_nr_n^\alpha
},
\qquad
\bar v_{i,n}(x)
:=
\frac{
	(\eta u_{i,n})(x_n+r_nx)
}{
	L_nr_n^\alpha
}.
\]
They are defined, in particular, on
\[
\Omega_n
:=
\frac{B_4-x_n}{r_n}
\subset
\frac{\Omega-x_n}{r_n}.
\]
Since $x_n\in B_2$ and $r_n\to0$, we have
\[
B_{1/r_n}\subset\Omega_n
\]
for all sufficiently large $n$. Thus the rescaled domains exhaust $\R^N$. In what follows, whenever a statement concerning $\mf v_n$ or $\bar{\mf v}_n$ is made on a fixed bounded subset of $\R^N$, it is understood to hold for all sufficiently large $n$.

The reason for introducing both rescalings is that they retain different useful features of the original sequence. The normalization of the H\"older quotient is naturally inherited by $\bar{\mf v}_n$, whereas the equations satisfied by this sequence contain additional terms generated by the cut-off function. Conversely, $\mf v_n$ preserves the structure of the original system after rescaling, but its H\"older seminorm is not directly controlled by the normalization. Notice moreover that $v_{i,n}(0)=\bar v_{i,n}(0)$.
The first lemma establishes basic properties of these two sequences, allowing also to relate one to the other.

\begin{lemma}\label{lemma: basic prop}

Using the previous notation, we have that:
\begin{itemize}

\item[($i$)] the sequence $\{\bar{\mf{v}}_n\}$ has uniformly bounded $\alpha$-H\"older semi-norm in $\Omega_n$, and in particular
\[
\sup_{i=1,\dots,d} \sup_{\substack{x \neq y \\ x,y \in \Omega_n}}    \frac{ |\bar v_{i,n}(x)-\bar v_{i,n}(y)|}{|x-y|^{\alpha}} = \frac{ |\bar v_{1,n}(0)-\bar v_{1,n}\left(\frac{y_n-x_n}{r_n}\right)|}{\left|\frac{y_n-x_n}{r_n}\right|^{\alpha}} = 1
\]
for every $n$.
\item[($ii$)] $v_{i,n}$ is a weak solution of 
\begin{equation}\label{eq: eq satisfied by v_n}
-\Delta v_{i,n} = -M_n v_{i,n} \sum_{\substack{J\subseteq [d]\setminus \{i\}\\ |J|=k-1}} \gamma_{J,i} v_{J,n}^{2}+g_{i,n}(x,v_{i,n}), \quad v_{i,n}\geq 0 \qquad \text{in $\Omega_n$},
\end{equation}
where
\[
 M_n:= \beta_n r_n^{2+2(k-1)\alpha} \left(\frac{L_n}{\eta(x_n)}\right)^{2(k-1)} \quad \text{ and } \quad g_{i,n}(x,u):=\frac{\eta(x_n) r_n^{2-\alpha}}{L_n}f_{i,\b_n}\left(x_n+r_n x,\frac{r_n^\alpha L_n}{\eta(x_n)}u\right).
\] 
Furthermore, 
\begin{equation}\label{eq: g_n to 0}
	g_{i,n}(x,v_{i,n}(x))\to 0 \quad \text{locally uniformly as $n\to\infty$ for every $i=1,\dots,d$.}
\end{equation}
\item[($iii$)] for every compact set $K \subset \R^N$, we have 
\[
\sup_K |\mf{v}_n - \bar{\mf{v}}_n | \to 0 \qquad \text{as $n \to \infty$},
\]
\item[($iv$)] for every compact set $K\subset \R^N$ and for every $n$ so large that $\Omega_n \supset K$,
\[
|v_{i,n}(x) - v_{i,n}(y)| \le o_n(1) + |x-y|^\alpha \quad \text{for every $x, y \in K$},
\]
where $o_n(1) \to 0$ uniformly on $x,y \in K$; in particular $\{v_{i,n}\}$ has uniformly bounded oscillation in any compact set.
\item[($v$)] $\mf{v}_n$ is a minimizer of \eqref{eq: eq satisfied by v_n} with respect to variations with compact support, namely for every $\Omega' \ssubset \R^N$
\[
J_{M_n,\bs g_n}(\mf{v}_n,\Omega') \le J_{M_n,\bs g_n}(\mf{v}_n+ \bs{\varphi},\Omega') \quad \forall \bs{\varphi} \in H^1_0(\Omega', \R^d),
\]
for sufficiently large $n$, where $J_{M_n,\bs g_n}(\cdot, \Omega)$ is defined as in \eqref{eq:def J_beta} (with the natural primitive $G_{i,n}(x,u)=\int_0^u g_{i,n}(x,s)\, ds$ appearing in the functional).
\end{itemize}
\end{lemma}
\begin{proof}
The proof of ($i$), \eqref{eq: eq satisfied by v_n} and ($v$) is a simple change of variables. To show \eqref{eq: g_n to 0}, observe that for every $K$ compact and $n$ large enough, $x_n+r_n K \subset B_3$, so for any $x\in K$ 
\begin{align*}
|g_{i,n}(x,v_{i,n}(x))|&=\frac{\eta(x_n) r_n^{2-\alpha}}{L_n}\left|f_{i,\b_n}\left(x_n+r_n x,u_{i,n}(x_n+r_n x)\right)\right|\\ 
&\leq \frac{\eta(x_n) r_n^{2-\alpha} R_n}{L_n}\leq \frac{\eta(x_n)r_n^{2-\alpha}}{n}\to 0 \quad \text{as } n\to\infty, 
\end{align*}
where in the last inequality we used \eqref{eq:L_n}.
As far as ($iii$) is concerned, since $\eta$ is globally Lipschitz continuous with constant denoted by $l$ and using \eqref{eq:L_n} again, we have for every $x\in K$ and $n$ large enough,
\[
|v_{i,n}(x) - \bar v_{i,n}(x)|= \frac{|u_{i,n}(x_n+r_n x)|}{L_n r_n^{\alpha}} |\eta(x_n)-\eta (x_n+r_n x)| \le \frac{l C r_n^{1-\alpha}}{n} |x|.
\]
Now, for any $x,y\in K$,
\begin{align*}
		|v_{i,n}(x)-v_{i,n}(y)|&=|v_{i,n}(x)-\bar v_{i,n}(x)+\bar v_{i,n}(x)-\bar v_{i,n}(y)+\bar v_{i,n}(y)-v_{i,n}(y)| \\
		&\leq 2\sup_K |\mf v_n-\bar{\mf v}_n|+|\bar v_{i,n}(x)-\bar v_{i,n}(y)|,	
\end{align*}
so ($iv$) follows by ($iii$) and ($i$).
\end{proof}

In the next statement we collect some useful properties that hold when one component is bounded at $0$.

\begin{lemma}\label{lemma: basic convergence when v(0) bounded}
Suppose that $\{v_{i,n}(0)\}$ is bounded. Then there exists a function $v_i \in C^{0}(\R^N) \cap H^1_{\loc}(\R^N)$, globally $\alpha$-H\"older continuous, such that, up to a subsequence: 
\begin{itemize}
\item[($i$)] $v_{i,n}, \bar v_{i,n} \to v_i$ 
locally uniformly on $\R^N$; 
\item[($ii$)] $v_i$ is non-constant in $B_1$ if $i=1$;
\item[($iii$)] for every $K \subset \R^N$ compact there exists $C>0$ such that
\[
M_n \int_{K} \Bigg(\sum_{\substack{J\subseteq [d]\setminus \{i\} \\ |J|=k-1}}\gamma_{J,i} v_{J,n}^2 \Bigg) v_{i,n}\,dx \le C
\]
\item[($iv$)] $v_{i,n} \to v_i$ strongly in $H^1_{\loc}(\R^N)$.
\end{itemize}
\end{lemma}

\begin{proof}
Point ($i)$ follows directly from Ascoli-Arzel\`a theorem and Lemma \ref{lemma: basic prop}. Point ($ii$) is a consequence of Lemma \ref{lemma: basic prop} $(i)$ and the local uniform convergence of $\bar v_{1,n}$ to $v_{1}$. Concerning point ($iii$), let us test the equation of $v_{i,n}$ with a cut-off function $\varphi \in C^\infty_c(\R^N)$. We obtain
\[
M_n \int_{\Omega_n} \Bigg(\sum_{\substack{J\subseteq [d]\setminus \{i\} \\ |J|=k-1}}\gamma_{J,i} v_{J,n}^2 \Bigg) v_{i,n} \varphi \,dx = \int_{\Omega_n} \left(v_{i,n} \Delta \varphi+g_{i,n}(x,v_{i,n}(x))\varphi\right) \,dx,
\]
whence the desired estimate follows (recall that $g_{i,n}\to 0$ locally uniformly). Finally, for point ($iv$), let us test the equation of $v_{i,n}$ with $v_{i,n} \varphi^2$, where $\varphi \in C^\infty_c(\R^N)$ is an arbitrary nonnegative cut-off function. We obtain
\[
\int_{\Omega_n} \Big(|\nabla v_{i,n}|^2 \varphi^2 +2v_{i,n} \varphi \nabla v_{i,n}\cdot \nabla \varphi +M_n \varphi^2 \sum_{\substack{J\subseteq [d] \\ |J|=k}} \gamma_J v_{J,n}^2 - g_{i,n}(x,v_{i,n}(x))v_{i,n}\varphi^2 \Big) dx = 0.
\]
This implies that
\[
\int_{\Omega_n} |\nabla v_{i,n}|^2 \varphi^2\,dx \le 4 \int_{\Omega_n}  v_{i,n}^2 |\nabla \varphi|^2\,dx+2\int_{\Omega_n} g_{i,n}(x,v_{i,n}(x))v_{i,n}\varphi^2 \, dx,
\]
which entails the boundedness of $\{v_{i,n}\}$ in $H^1_{\loc}(\R^N)$, and hence up to a subsequence $v_{i,n} \weak v_i$ weakly in $H^1_{\loc}(\R^N)$. To show the strong convergence, we test the equation of $v_{i,n}$ with $(v_{i,n}-v_i) \varphi$, where $\varphi \in C^\infty_c(\R^N)$ is an arbitrary nonnegative cut-off function, and observe that
\begin{multline*}
	\left|\int_{\Omega_n} \varphi \nabla v_{i,n} \cdot \nabla (v_{i,n}-v_i)  \, dx \right|\\
	\le \|v_{i,n}-v_i\|_{L^\infty(\mathrm{supp}(\varphi))} \int_{\Omega_n} \Big( |\nabla v_{i,n} \cdot \nabla \varphi| + M_n \varphi \, v_{i,n} \sum_{\substack{J\subseteq [d]\setminus \{i\} \\ |J|=k-1}} \gamma_J v_{J,n}^2 + |g_{i,n}(x,v_{i,n})| \varphi\Big)dx.
\end{multline*}
The previous points ensure that the right hand side tends to $0$ as $n \to \infty$. Therefore
\[
\lim_n \int_{\Omega_n} \varphi \left( |\nabla v_{i,n}|^2- |\nabla v_i|^2\right) dx =\lim_{n} \int_{\Omega_n} \varphi \nabla v_{i,n} \cdot \nabla (v_{i,n}-v_i)\,dx =0,
\]
that by the arbitrariness of $\varphi$ implies the convergence of the $L^2$ norms of the gradients on every compact subset, and this together with weak $H^1_{\loc}$ convergence gives strong $H^1_{\loc}$ convergence.
\end{proof}

The previous lemma shows that the boundedness of $\{v_{i,n}(0)\}$ entails several consequences. On the other hand, we do not know whether this property holds or not. If not, we will often use the following property for a suitably normalized sequence.

\begin{lemma}\label{lemma: properties w_i}
Let 
\[
\mf{w}_n(x):= \mf{v}_n(x)-\mf{v}_n(0), \qquad \bar{\mf{w}}_n(x):= \bar{\mf{v}}_n(x)-\bar{\mf{v}}_n(0).
\]
Then:
\begin{itemize}
	\item[$(i)$] there exists a function $\mf{w}$, globally $\alpha$-H\"older continuous in $\R^N$, such that $\mf{w}_n, \bar{\mf{w}}_n \to \mf{w}$ locally uniformly in $\R^N$. Moreover, $w_1$ is non-constant.
	\item[$(ii)$] if in addition $\Delta w_{i,n} \to \lambda$ in $L^1_{\loc}(\R^N)$ for an index $i$ and some $\lambda\geq 0$, then $w_{i,n} \to w_i$ strongly in $H^1_{\loc}(\R^N)$ and $\Delta w_i = \lambda$.
\end{itemize}

\end{lemma}
\begin{proof}
$(i)$ The local uniform convergence $\bar{\mf{w}}_n \to \mf{w}$ follows directly by the Ascoli-Arzel\`a theorem, the uniform bound on the $C^{0,\alpha}$ seminorm of $\bar{\mf{w}}_n$ and the fact that $\mf w_n(0)=\bar{\mf w}_n(0)=\mf 0$. Moreover the $\alpha$-H\"older seminorm of $\mf{w}$ is equal to $1$, and $w_1$ is non-constant. The fact that also $\mf{w}_n$ converges to $\mf{w}$ is a consequence of Lemma \ref{lemma: basic prop}-($iii$).

\medskip
$(ii)$ We first check that $\{w_{i,n}\}$ is bounded in $H^1(B_r)$: for any $\varphi \in C^\infty_c(B_{2r})$, we have
\[
\int_{B_{2r}} |\nabla w_{i,n}|^2 \varphi^2\,dx = - 2 \int_{B_{2r}} \varphi w_{i,n} \nabla w_{i,n} \cdot \nabla \varphi\,dx -  \int_{B_{2r}} w_{i,n} \varphi^2 \Delta w_{i,n}\,dx,
\]
whence
\[
\int_{B_{2r}} |\nabla w_{i,n}|^2 \varphi^2\,dx \le C  \int_{B_r}  w_{i,n}^2 |\nabla \varphi|^2\,dx + C \|w_{i,n}\|_{L^\infty({B_r})} \|\Delta w_{i,n}\|_{L^1({B_r})},
\]
which gives the boundedness in $H^1({B_r})$ by local uniform convergence of $\{w_{i,n}\}$ and $L^1_{\loc}$ convergence of $\Delta w_{i,n}$. By arbitrariness of $r$, up to a subsequence $w_{i,n} \weak w_i$ weakly in $H^1_{\loc}(\R^N)$. 
Notice that $\Delta w_{i,n}=\Delta v_{i,n}$ and for any $\varphi\in C^\infty_c(\R^N)$ we have, multiplying the equation of $v_{i,n}$ by $\varphi$,
\begin{equation}\label{eq:2}
	\int_{\Omega_n} \Big(M_n \varphi v_{i,n} \sum_{\substack{J\subseteq [d]\setminus \{i\} \\ |J|=k-1}} \gamma_{J,i} v_{J,n}^2\Big)\,dx = \int_{\Omega_n} \left(\Delta w_{i,n} + g_{i,n}(x,v_{i,n})\right) \varphi \, dx\leq C,
\end{equation}
where $C$ is independent of $n$. Then, multiplying the equation of $w_{i,n}$ by $(w_{i,n}-w_i)\varphi$, where $\varphi\in C^\infty_c(\R^N)$ is an arbitrary nonnegative cut-off function, we obtain 
\[
\begin{split}
\left|\int_{\Omega_n} \varphi \nabla w_{i,n} \cdot \nabla (w_{i,n}-w_i)  \, dx\right| &\le \|w_{i,n}-w_i\|_{L^\infty(\mathrm{supp}(\varphi))} \int_{\Omega_n} |\nabla w_{i,n} \cdot \nabla \varphi| \, dx \\
&+\|w_{i,n}-w_i\|_{L^\infty(\mathrm{supp}(\varphi))} \int_{\Omega_n} \Big(M_n \varphi \, v_{i,n} \sum_{\substack{J\subseteq [d]\setminus \{i\} \\ |J|=k-1}} \gamma_{J,i} v_{J,n}^2+|g_{i,n}(x,v_{i,n})|\varphi \Big)dx,
\end{split}
\]
where the right hand side goes to $0$ because of the convergence $w_{i,n} \to w_i$ in $L^\infty_{\loc}(\R^N)$ and \eqref{eq:2}. Combined with the weak convergence, this proves strong convergence in $H^1_{\loc}(\R^N)$.
The fact that $\Delta w_i = \lambda$ follows now by taking the limit in the equation 
\[
\int_{\Omega_n} \nabla w_{i,n}\cdot \nabla \varphi\,dx = \int_{\Omega_n} (-\Delta w_{i,n})\varphi\,dx  \quad \forall \varphi \in C^\infty_c(\R^N). \qedhere
\]

\end{proof}

We now proceed with the analysis of the possible asymptotic behaviours of the components at the origin. A key novelty of our proof with respect to the existing literature lies precisely in the strategy adopted at this stage. Indeed, when both the order of the interaction $k$ and the number of components $d$ are arbitrary, the number of possible asymptotic configurations grows considerably, making a direct case-by-case analysis impractical. To overcome this difficulty, we group the components according to their asymptotic behaviour and develop a systematic argument based on the cardinalities of the resulting classes.

More precisely, after passing to a subsequence, we partition the set of components according to whether their value at the origin remains bounded or diverges, and define
\begin{align*}
	I_{\infty}
	&:=
	\left\{
	i\in[d]:
	v_{i,n}(0)\to+\infty
	\right\},\\
	I_F
	&:=
	[d]\setminus I_{\infty}
	=
	\left\{
	i\in[d]:
	\{v_{i,n}(0)\}_n \text{ is bounded}
	\right\}.
\end{align*}
We will analyze the possible cardinalities of these sets and, relying on the rigidity and Liouville-type results established in the previous sections together with the minimality properties of the original sequence, rule them out one by one, eventually reaching the desired contradiction.

\medskip

The following lemma provides a first useful control of the interaction term in the presence of an unbounded component.
\begin{lemma}\label{lemma: 1 divergent implies integral bound}
Suppose that up to a subsequence $v_{i,n}(0) \to +\infty$. Then for every $r>0$ there exists $C=C(r)>0$ such that 
\[
M_n v_{i,n}(0) \int_{B_r} \sum_{\substack{J\subseteq [d]\setminus \{i\} \\ |J|=k-1}} \gamma_{J,i} v_{J,n}^2\,dx \le C.
\] 
\end{lemma}

\begin{proof}
We define
\[
\begin{split}
H_{i,n}(r) &:= \frac1{r^{N-1}} \int_{S_r} v_{i,n}^2\,d\sigma, \\
E_{i,n}(r) & := \frac1{r^{N-2}} \int_{B_r} \Big( |\nabla v_{i,n}|^2 + M_n \sum_{\substack{J\subseteq [d] \\ |J|=k \\ i\in J}} \gamma_J v_{J,n}^2\Big)\,dx.
\end{split}
\] 
By multiplying the equation of $v_{i,n}$ by $v_{i,n}$, and integrating on $B_r$, we see that
\[
\int_{B_r} \Big(|\nabla v_{i,n}|^2 + M_n \sum_{\substack{J\subseteq [d] \\ |J|=k \\ i\in J}} \gamma_J v_{J,n}^2-g_{i,n}(x,v_{i,n}(x))v_{i,n}\Big)dx = \int_{S_r} v_{i,n} \pa_\nu v_{i,n}\,d\sigma,
\]
whence
\[
H_{i,n}'(r) = \frac2{r} E_{i,n}(r)-\frac{2}{r^{N-1}}\int_{B_r}g_{i,n}(x,v_{i,n}(x))v_{i,n}\, dx.
\] 
Therefore, integrating from $r/2$ to $r$ we obtain
\begin{equation}\label{H_i' e E_i}
H_{i,n}(r) - H_{i,n}\left(\frac{r}{2}\right) + \int_{r/2}^{r}\dfrac{2}{s^{N-1}}\left(\int_{B_s} g_{i,n}(x,v_{i,n}(x))v_{i,n}\, dx\right)ds = \int_{r/2}^r \frac2{s} E_{i,n}(s)\,ds.
\end{equation}
Now, the left hand side can be estimated from above thanks to properties ($i$),($ii$) and ($iv$) in Lemma \ref{lemma: basic prop}:
\begin{multline*}
	H_{i,n}(r) - H_{i,n}\left(\frac{r}{2}\right) + \int_{r/2}^{r}\dfrac{2}{s^{N-1}}\left(\int_{B_s} g_{i,n}(x,v_{i,n}(x))v_{i,n}\, dx\right)ds \\
	\leq \int_{S_1} \left( v_{i,n}^2(ry)- v_{i,n}^2 \left(\frac{r}{2}y\right)\right)d\sigma(y) + C_1(r)\int_{B_r} g_{i,n}(x,v_{i,n}(x))v_{i,n}\\
	\leq C_2(r) \left( v_{i,n}(0)+1\right) \le C_3(r) v_{i,n}(0),
\end{multline*}
since $v_{i,n}(0)\to+\infty$. Moreover, the right hand side of \eqref{H_i' e E_i} can be estimated from below as follows:
\[
\begin{split}
\int_{r/2}^r \frac2{s} E_{i,n}(s)\,ds & \ge \frac{r}{2} \min_{s \in \left[\frac{r}{2},r\right]} \frac{2}{s} E_{i,n}(s) \\
& \ge r M_n \min_{s \in \left[\frac{r}{2},r\right]} \frac{1}{s^{N-1}} \int_{B_s} \sum_{\substack{J\subseteq [d] \\ |J|=k \\ i\in J}} \gamma_J v_{J,n}^2\,dx\\
& \ge C_4(r) M_n v_{i,n}^2(0) \int_{B_{r/2}} \frac{v_{i,n}^2}{v_{i,n}^2(0)} \sum_{\substack{J\subseteq [d]\setminus \{i\} \\ |J|=k-1 }} \gamma_{J,i} v_{J,n}^2\,dx.
\end{split}
\]
Since $v_{i,n}(0) \to +\infty$ and the oscillation of $v_{i,n}$ is locally uniformly bounded, we have that $v_{i,n}/v_{i,n}(0) \to 1$ uniformly on $B_{r/2}$ as $n \to \infty$. Hence, coming back to \eqref{H_i' e E_i}, we finally infer that 
\[
M_n v_{i,n}^2(0) \int_{B_{r/2}} \sum_{\substack{J\subseteq [d]\setminus \{i\} \\ |J|=k-1 }} \gamma_{J,i} v_{J,n}^2\,dx \le C_5(r) v_{i,n}(0),
\]
whence the thesis follows.
\end{proof}

\begin{lemma}\label{lemma: not d divergent}
It is not possible that $v_{i,n}(0) \to +\infty$ for every $i=1,\dots,d$.
\end{lemma}

\begin{proof}
By contradiction, we suppose that $v_{i,n}(0) \to +\infty$ for every $i$. Then, since $v_{i,n}$ has locally bounded oscillations, $v_{i,n}\to +\infty$ locally uniformly. Moreover,
\begin{equation}\label{eq:3}
\left|\frac{v_{i,n}(x)}{v_{i,n}(0)} - 1 \right|=\left|\frac{v_{i,n}(x)-v_{i,n}(0)}{v_{i,n}(0)}\right|\to 0 \quad \text{locally uniformly for every $i$}.
\end{equation}
Thus, by Lemma \ref{lemma: 1 divergent implies integral bound} 
\begin{equation}\label{eq:4}
\begin{split}
M_n v_{1,n}(0) \sum_{\substack{J\subseteq \{2,\dots,d\} \\ |J|=k-1}} \gamma_{J,1} v_{J,n}(0)^2 & \le C M_n v_{1,n}(0) \sum_{\substack{J\subseteq \{2,\dots,d\} \\ |J|=k-1}} \gamma_{J,1} v_{J,n}(0)^2  \int_{B_1} \frac{v_{J,n}^2}{v_{J,n}(0)^2} dx \\
&= C M_n v_{1,n}(0) \int_{B_1} \sum_{\substack{J\subseteq \{2,\dots,d\} \\ |J|=k-1}} \gamma_{J,1} v_{J,n}^2\,dx \le C',
\end{split}
\end{equation}
with $C,C'>0$ independent of $n$; namely, the interaction term in the equation of $v_{1,n}$ is bounded at $0$. Furthermore,
\begin{equation}\label{eq:18}
\begin{split}
\Big| M_n v_{1,n}(x) \sum_{\substack{J\subseteq \{2,\dots,d\} \\ |J|=k-1}}\gamma_{J,1} v_{J,n}^2(x) & - M_n v_{1,n}(0) \sum_{\substack{J\subseteq \{2,\dots,d\} \\ |J|=k-1}}\gamma_{J,1} v_{J,n}^2(0) \Big| \\
& = M_n v_{1,n}(0) \sum_{\substack{J\subseteq \{2,\dots,d\} \\ |J|=k-1}}\gamma_{J,1} v_{J,n}^2(0) \left| 1- \frac{v_{1,n}(x)}{v_{1,n}(0)} \frac{v_{J,n}^2(x)}{v_{J,n}^2(0)} \right| \to 0
\end{split}
\end{equation}
as $n \to \infty$, locally uniformly, by \eqref{eq:3} and \eqref{eq:4}. Therefore, there exists a constant $\lambda \ge 0$ such that, up to a subsequence, 
\[
M_n v_{1,n}(x) \sum_{\substack{J\subseteq \{2,\dots,d\} \\ |J|=k-1}}\gamma_{J,1} v_{J,n}^2(x) \to \lambda \quad \text{locally uniformly in $\R^N$}.
\]
Recall that by Remark \ref{rmk: regularity beta fixed} $v_{1,n}$ (and thus also $w_{1,n}$) is in $W^{2,p}$ for every $p$, so \eqref{eq: eq satisfied by v_n} holds pointwise a.e..
Then, together with \eqref{eq: g_n to 0}, we obtain $\Delta w_{1,n}\to \lambda$ in $L^1_{\loc}(\R^N)$.
By Lemma \ref{lemma: properties w_i}, $w_{1,n}$ converges locally uniformly to a limit function $w_1$, non-constant and globally H\"older continuous, which moreover satisfies 
\[
\Delta w_1 = \lambda \quad \text{in $\R^N$},
\]
contradicting Proposition \ref{thm: Liouville 1 comp}. 

\end{proof}
\begin{lemma}\label{lemma: k-1 diverg implies bound}
	Suppose that up to a subsequence $v_{i_1,n}(0)\to+\infty, \dots, v_{i_{k-1},n}(0) \to +\infty$ for a $(k-1)$-tuple of indexes $i_1, \dots, i_{k-1} \in \{1,\dots,d\}$. Then 
	\[
	M_n \prod_{\ell=1}^{k-1} v_{i_\ell,n}^2(0) \le C
	\]
	for some $C>0$ independent of $n$. 
\end{lemma}

\begin{proof}
	We argue by contradiction, and suppose that up to a subsequence $M_n \prod_{\ell=1}^{k-1} v_{i_\ell,n}^2(0) \to +\infty$ as $n \to \infty$. 
	
	\smallskip
	
	\emph{Step 1)} For any $i\neq i_1,\dots,i_{k-1}$ and for any $r>1$, we claim that $\|v_{i,n}\|_{L^\infty(B_{r})} \le C$ for a positive constant $C>0$ depending only on $r$. 
	
	\smallskip
	
	First of all, since $v_{j,n}$ has locally bounded oscillation for any $j$ by Lemma \ref{lemma: basic prop}, we have that $v_{j,n} \to +\infty$ locally uniformly for $j=i_1,\dots,i_{k-1}$. Moreover, there exists $C(r)>0$ depending on $r$ but not on $n$ such that
	\begin{equation*}
		I_n := M_n \inf_{B_{r}}\prod_{\ell=1}^{k-1} v_{i_\ell,n}^2  \ge M_n \prod_{\ell=1}^{k-1}\left(v_{i_\ell,n}(0)- C(r)\right)^2 \to +\infty.
	\end{equation*}
	Let us test the equation of $v_{i,n}$ with $v_{i,n} \varphi^2$, with $\varphi \in C^\infty_c(\R^N)$ such that $0 \le \varphi \le 1$, $\varphi \equiv 1$ in $B_r$, and $\supp(\varphi) \subset B_{2r}$: we obtain 
	\[
	\int_{\Omega_n} \Big(|\nabla v_{i,n}|^2 \varphi^2 +2v_{i,n} \varphi \nabla v_{i,n}\cdot \nabla \varphi +M_n \varphi^2 \sum_{\substack{J\subseteq [d] \\ |J|=k \\ i\in J}} \gamma_J v_{J,n}^2-g_{i,n}(x,v_{i,n}(x))v_{i,n}\varphi^2 \Big) dx = 0,
	\]
	whence
	\begin{equation}\label{stima prod}
		\begin{split}
			C(r) \, I_n \inf_{B_r} v_{i,n}^2 & \le M_n \int_{B_r} \sum_{\substack{J\subseteq [d] \\ |J|=k \\ i\in J}}\gamma_J v_{J,n}^2 \,dx  \le \int_{\Omega_n} \Big(\frac12 |\nabla v_{i,n}|^2 \varphi^2 +M_n \varphi^2\sum_{\substack{J\subseteq [d] \\ |J|=k \\ i\in J}}\gamma_J v_{J,n}^2 \Big) dx\\
			& \le 2\int_{\Omega_n} \left(v_{i,n}^2 |\nabla \varphi|^2+g_{i,n}(x,v_{i,n}(x))v_{i,n}\varphi^2\right)\,dx
			\le C(r) (\sup_{B_{2r}} v_{i,n}^2+\sup_{B_{2r}} v_{i,n}),
		\end{split}
	\end{equation}
	for some $C(r)>0$. We also observe that $\inf_{B_r} v_{i,n} \ge \sup_{B_{2r}} v_{i,n} - C(r)$, with $C(r)>0$ (once again, we used the fact that $\mf{v}_n$ has locally bounded oscillation). Therefore, \eqref{stima prod} yields
	\[
	I_n \sup_{B_{2r}} v_{i,n}^2 \le C(r) \Big(\sup_{B_{2r}} v_{i,n}^2 +  I_n \sup_{B_{2r}} v_{i,n}+\sup_{B_{2r}}v_{i,n}\Big).
	\]
	Since $I_n\to+\infty$, if $\sup_{B_{2r}} v_{i,n}^2\to+\infty$, then we would obtain
	\begin{equation*}
		1\leq C(r)\left(\frac{1}{I_n}+\frac{1}{\sup_{B_{2r}}v_{i,n}}+\frac{1}{I_n \sup_{B_{2r}}v_{i,n}}\right)=o(1),
	\end{equation*}
	which is a contradiction. Then the claim is proved.
	
	\smallskip
	
	\emph{Step 2)} $1\in\{i_1,\dots,i_{k-1}\}$
	
	\smallskip
	
	Let $i\neq i_1,\dots,i_{k-1}$. By definition of $I_n$, in $B_{r}$ we have, for every $r>0$,
	\begin{align*}
		-\Delta v_{i,n} &= - M_n \gamma_{i,i_1,\dots,i_{k-1}}v_{i,n} \prod_{\ell=1}^{k-1} v_{i_\ell,n}^2 - M_n v_{i,n} \sum_{\substack{J\subseteq [d]\setminus \{i\} \\ |J|=k-1 \\ J\neq\{i_1,\dots,i_{k-1}\}}}\gamma_{J,i}v_{J,n}^2+g_{i,n}(x,v_{i,n}(x)) \\
		&\le - I_n v_{i,n}+g_{i,n}(x,v_{i,n}(x)).
	\end{align*}
	By Lemma \ref{lemma:decadimento esponenziale} and Step 1, and recalling that $I_n \to +\infty$, we obtain
	\begin{equation}\label{eq:1}
		\|v_{i,n}\|_{L^2(B_{r/2})} \le C \sup_{B_{r}} v_{i,n} e^{- \frac{r\sqrt{I_n}}{2}}+\frac{\|g_{i,n}(\,\cdot\,,v_{i,n}(\,\cdot\,))\|_{L^2(B_{r/2})}}{I_n} \to 0 
	\end{equation}
	as $n \to \infty$. Since $r>1$ was arbitrarily chosen, $v_{i,n} \to 0$ in $L^2_{\loc}(\R^N)$. Together with local uniform convergence to $v_i$ provided by Lemma \ref{lemma: basic convergence when v(0) bounded}, this implies $v_i=0$. By Lemma \ref{lemma: basic prop} again, $v_1$ is not constant, hence $i\neq 1$ and Step 2 is proved.
	
	\smallskip
	\emph{Step 3: final contradiction}
	
	By Step 2, without loss of generality we can assume $\{i_1,\dots,i_{k-1}\}=\{1,\dots,k-1\}$, so that $v_{j,n}$ is bounded locally uniformly for all $k\leq j \leq d$ by Step 1.
	Let now $r>1$, and $\bar x_n \in \overline{B_r}$ be such that
	\[
	M_n v_{1,n}^2\dots v_{k-1,n}^2\,(\bar x_n) = I_n.
	\]
	Since $\{\mf{v}_n\}$ has locally bounded oscillation and $v_{1,n},\dots, v_{k-1,n} \to +\infty$ locally uniformly, we have that for any $x\in B_r$,
	\[
	M_n \prod_{\ell=1}^{k-1}v_{\ell,n}(x)^2 \le M_n \prod_{\ell=1}^{k-1} (v_{\ell,n}(\bar x_n)+C(r))^2 \le 2 M_n \prod_{\ell=1}^{k-1}v_{\ell,n}(\bar x_n)^2 = 2I_n.
	\]
	Then we deduce that
	\begin{align*}
		\|\Delta v_{1,n}\|_{L^1(B_r)} &\leq M_n \sum_{\substack{J\subseteq [d]\setminus \{1\} \\ |J|=k-1}} \gamma_{J,1} \|v_{1,n} v_{J,n}^2\|_{L^1(B_r)} + \|g_{1,n}(\,\cdot\,,v_{1,n}(\,\cdot\,))\|_{L^1(B_r)} \\
		&\leq M_n \sum_{\substack{J\subseteq [d]\setminus \{1\} \\ |J|=k-1}} \gamma_{J,1} \|v_{1,n}^2 v_{J,n}^2\|_{L^1(B_r)} + \|g_{1,n}(\,\cdot\,,v_{1,n}(\,\cdot\,))\|_{L^1(B_r)} \\
		&\leq C(r) I_n \left(e^{-r\sqrt{I_n}}+\frac{\|g_{1,n}(\,\cdot\,,v_{1,n}(\,\cdot\,))\|_{L^1(B_r)}}{I_n}\right)+\|g_{1,n}(\,\cdot\,,v_{1,n}(\,\cdot\,))\|_{L^1(B_r)}\to 0,
	\end{align*}
	where in the last inequality we used \eqref{eq:1} and the fact that any $k$-tuple contains at least one component with index $j\geq k$. 
	If we consider $w_{1,n}:= v_{1,n}-v_{1,n}(0)$, Lemma \ref{lemma: properties w_i}, implies that $w_1$ is globally H\"older continuous, non-constant, and harmonic, contradicting Proposition \ref{thm: Liouville 1 comp} $(i)$.
\end{proof}
\begin{lemma}\label{lemma: tra k-1 e d-1 diverg}
It is not possible that $k-1\leq |I_\infty|\leq d-1$.
\end{lemma}

\begin{proof}
Assume by contradiction that $k-1\leq |I_\infty|\leq d-1$. Then, for any $J\subseteq I_\infty$ with $|J|\leq k-2$, we can find $J'\subseteq I_\infty$ with $J\subset J'$ and $|J'|=k-1$. Then
\begin{equation}\label{eq:5}
	M_n \prod_{i\in J} v_{i,n}(0)^2\leq C \frac{M_n\prod_{i\in J'}v_{i,n}(0)^2}{\prod_{i\in J'\setminus J}v_{i,n}(0)^2}\to 0 \quad \text{as $n\to+\infty$},
\end{equation}
since the numerator is bounded by Lemma \ref{lemma: k-1 diverg implies bound}, while the denominator diverges. An analogous reasoning proves that $M_n\to 0$.
Summarizing, since $v_{i,n}$ has locally bounded oscillations for every $i$, by Lemma \ref{lemma: k-1 diverg implies bound} up to subsequences there holds:
\begin{enumerate}
	\item[$a)$] $\forall J\subseteq I_\infty$, $|J|=k-1$, $M_n v_{J,n}^2\to \lambda_J\geq 0 $ locally uniformly in $\R^N$;
	\item[$b)$] $\forall J\subseteq I_\infty$, $|J|\leq k-2$, $M_n v_{J,n}^2\to 0 $ locally uniformly in $\R^N$.
\end{enumerate}

Then, for every $i\in I_F$ (notice that $I_F$ is non empty as $|I_\infty|\leq d-1$), we have
\begin{equation*}
	\Delta v_{i,n} = M_n v_{i,n} \sum_{\substack{J\subseteq I_\infty \\ |J|=k-1}} \gamma_{J,i} v_{J,n}^2 + M_n v_{i,n} \sum_{\substack{J\subseteq I_\infty, J'\subseteq I_F \\ |J|\leq k-2 \\ |J|+|J'|=k-1 }} \gamma_{J,J',i}v_{J,n}^2  v_{J',n}^2-g_{i,n}(x,v_{i,n}(x)).
\end{equation*}
By what we have just proved, local uniform boundedness of $v_{j,n}$ for $j\in I_F$ and \eqref{eq: g_n to 0}, it follows that the right hand side tends to $\lambda v_i$ locally uniformly, with $$\lambda=\sum_{\substack{J\subseteq I_\infty \\ |J|=k-1}} \lambda_J.$$ 
Then, by Lemma \ref{lemma: basic convergence when v(0) bounded}, $v_i$ solves $\Delta v_i=\lambda v_i$, and thus it is constant by Proposition \ref{thm: Liouville 1 comp}. By Lemma \ref{lemma: basic convergence when v(0) bounded} $(ii)$, this yields $1\in I_\infty$. 

Without loss of generality, we can suppose that $d\in I_F$. Then 
\[
\Delta v_{1,n} = M_n v_{1,n} \sum_{\substack{J\subseteq I_\infty\setminus \{1\} \\ |J|=k-1}} \gamma_{J,1} v_{J,n}^2 + M_n v_{1,n} \sum_{\substack{J\subseteq [d]\setminus \{1\} \\ |J|=k-1 \\ J\cap I_F \neq \emptyset}} \gamma_{J,1} v_{J,n}^2 - g_{1,n}(x,v_{1,n}(x)).
\]
Notice that:
\begin{itemize}
	\item each term of the form $M_n \gamma_{J,1} v_{1,n} v_{J,n}^2$ with $J\subseteq I_\infty\setminus\{1\}$ and $|J|=k-1$ (if any) converges to some $\lambda_J\geq 0$, arguing as in Lemma \ref{lemma: not d divergent} (cfr. \eqref{eq:4} and \eqref{eq:18});
	\item each term of the form $M_n \gamma_{J,1} v_{1,n} v_{J,n}^2$ with $J\cap I_F \neq \emptyset$ converges to $0$ locally uniformly since, for any $R>0$ and $x\in B_R$,
	\begin{equation*}
		M_n v_{1,n}(x) v_{J,n}^2(x)\leq \frac{M_n v_{1,n}^2(x) v_{J,n}^2(x)}{\inf_{B_R} v_{1,n}}\leq \frac{C}{\inf_{B_{R}} v_{1,n}} \to 0,
	\end{equation*}
	because of $a)$ and $b)$ above.
\end{itemize}
Then, recalling \eqref{eq: g_n to 0}, we obtain $\Delta v_{1,n}\to \lambda$ in $L^1_{\loc}(\R^N)$, so by Lemma \ref{lemma: properties w_i} $w_{1,n}$ converges to $w_1$ satisfying $\Delta w_1=\lambda$ where $$\lambda=\sum_{\substack{J\subseteq I_\infty\setminus\{1\} \\ |J|=k-1}} \lambda_J \geq 0.$$ Since moreover $w_1$ is also globally H\"older continuous and non-constant, Proposition \ref{thm: Liouville 1 comp} gives the desired contradiction.
%
%
%
%
%
%
\end{proof}

In what follows we focus on ruling out the case in which the number of diverging components is between $1$ and $k-2$. Notice that so far we have exploited the equation satisfied by $\mf v_n$, namely the \emph{criticality} of these functions with respect to the functional $J_{M_n,\bs g_n}$, but not their \emph{minimality}. The hypothesis of minimality is needed only in the following case, in analogy to what happens in the case $k=d=3$ studied in \cite{SoTe1}, where minimality is needed only to rule out the case of exactly one diverging component.

The next lemma shows the consequence of the minimality of $\mf{v}_n$ that is crucial in the following.
\begin{lemma}\label{lemma: minimality implies positivity} 
Suppose that there exist $x_0 \in \R^N$ and $R>\rho>0$ such that: 
\begin{itemize}
	\item[$i)$] $v_{i,n} \to v_i$ uniformly in $B_R(x_0)$ and in $H^1(B_R(x_0))$, with $v_i|_{S_{\rho}(x_0)} \not \equiv  0$, for some index $i$;
	\item[$ii)$] $d-(k-1)$ components converge to $0$ uniformly in $B_R(x_0)$ and in $H^1(B_R(x_0))$.
\end{itemize}
Then $v_i >0$ in $B_\rho(x_0)$.
\end{lemma}

\begin{proof}
Without loss of generality, we fix $x_0=0$ and we assume $i=d$ and $v_1,\dots, v_{d-k+1}$ converge to $0$ (in this proof we never use that $v_1$ is non-constant, hence the same proof works for every possible choice of indexes). We suppose by contradiction that $\{v_d=0\} \cap B_\rho \neq \emptyset$, and we aim at contradicting the minimality of $\mf{v}_n$, Lemma \ref{lemma: basic prop}-($v$).

Let 
\[
c_n:= J_{M_n,\bs g_n}(\mf{v}_n, B_R) = \min \left\{ J_{M_n,\bs g_n}(\mf{u}, B_R) : \mf{u} \in H^1(B_R, \R^d) \text{ such that } \mf{u}-\mf{v}_n \in H_0^1(B_R, \R^d)\right\}.
\]
Let $\eta \in C^1(\overline{B_R})$ be a radial cut-off function with the following properties: $0 \le \eta \le 1$, $\eta$ is radially increasing, $\eta \equiv 0$ in $B_\rho$, $\eta =1$ on $S_{R}$. We define a competitor $\tilde{{v}}_n$ by defining
\[
\begin{split}
	\tilde v_{j,n}&:= \eta v_{j,n} \quad \text{ for $j=1,\dots,d-k+1$}, \\
	\tilde v_{j,n}&:= v_{j,n} \qquad \text{for $j=d-k+2,\dots, d-1$}, \\
	\tilde v_{d,n}&:= \begin{cases} v_{d,n} &\text{in $B_R \setminus B_\rho$} \\ \text{harmonic extension of $v_{d,n}$ on $S_\rho$} & \text{in $B_\rho$}
	\end{cases}.
\end{split}
\]
By continuity of the harmonic extension, we have that $\tilde v_{d,n} \to \tilde v_d$ in $H^1(B_R) \cap C(\overline{B_R})$ as $n\to\infty$, where 
\[
\tilde v_{d} := \begin{cases} v_{d} & \text{in $B_R \setminus B_\rho$} \\ \text{harmonic extension of $v_{d}$ on $S_\rho$} & \text{in $B_\rho$}.
\end{cases}
\]
Since $v_d|_{S_\rho} \not \equiv 0$, by the strong maximum principle $\tilde v_{d}>0$ in $B_\rho$, and hence $\tilde v_{d} \neq v_{d}$. This implies, by the Dirichlet principle, that
\[
\int_{B_\rho}|\nabla \tilde v_d|^2\,dx < \int_{B_\rho}|\nabla v_d|^2\,dx,
\]
and thus, by $H^1$ convergence, there exists $\delta>0$ such that
\begin{equation}\label{diff 1}
\int_{B_R}\left(|\nabla \tilde v_{d,n}|^2 - |\nabla v_{d,n}|^2\right)\,dx \le  - \delta,
\end{equation}
for every $n$ large enough. 
Moreover, by definition of $G_{j,n}$ in Lemma \ref{lemma: basic prop} $(v)$, 
$$
|G_{j,n}(x,u)|=\left|\int_0^u g_{j,n}(x,s) \, ds\right| \leq |u|\sup_{s\in [0,u]} |g_{j,n}(x,s)|=\eta(x_n)\frac{r_n^{2-\alpha}}{L_n}|u|\sup_{s\in [0,u]} \left|f_{j,n}(x_n+r_n x,\frac{L_n r_n^{\alpha}}{\eta(x_n)}s)\right|
$$
Therefore, for every $x\in B_R$,
$$
|G_{j,n}(x,v_{j,n}(x))|\leq\eta(x_n)\frac{r_n^{2-\alpha}}{L_n} |v_{j,n}(x)|\sup \{\left|f_{j,n}(x_n+r_n x,s)\right|:s\in [0,u_{j,n}(x_n+r_n x)]\},
$$
and for $n$ large enough, recalling \eqref{eq:L_n},
$$
\sup_{x\in B_R}\sup \{\left|f_{j,n}(x_n+r_n x,s)\right|:s\in [0,u_{j,n}(x_n+r_n x)]\}<\frac{L_n}{n},
$$
so that for $j=1,\dots, d-k+1$, recalling that $r_n\to 0^+$ and $v_{j,n}$ is uniformly convergent, we obtain 
\begin{equation}\label{eq:termine g_j,n va a zero}
	\int_{B_R} G_{j,n}(x,v_{j,n}(x)) \, dx \to 0 \qquad \text{ as } n\to\infty,
\end{equation}
Similarly,
\begin{equation}\label{eq:termine g_d,n va a zero 1}
	\int_{B_\rho} G_{d,n}(x,v_{d,n}(x)) \, dx \to 0, \qquad \int_{B_R} G_{j,n}(x,\tilde v_{j,n}(x)) \, dx \to 0 \qquad \text{ as } n\to\infty,
\end{equation}
for $j=1,\dots, d-k+1$ by the same argument, since $\tilde v_{j,n}\leq v_{j,n}$. Finally observe that $v_{d,n}$ and $\tilde v_{d,n}$ differ only in $B_\rho$ and by the maximum principle
$$
\tilde v_{d,n}(x)\leq \max_{S_\rho} v_{d,n} \quad \forall x \in B_\rho,
$$
so arguing as before, recalling \eqref{eq:L_n}, for every $x\in B_\rho$ and $n$ large enough,
\begin{align*}
|G_{d,n}(x,\tilde v_{d,n}(x))|&\leq\eta(x_n)\frac{r_n^{2-\alpha}}{L_n} |\tilde v_{d,n}(x)|\sup \{\left|f_{d,n}(x_n+r_n x,s)\right|:s\in [0,\max_{S_{\rho r_n}(x_n)} u_{d,n}]\} \\
&\leq \frac{\eta(x_n)r_n^{2-\alpha}}{L_n}|\tilde{{v}}_{d,n}(x)| R_n<\frac{\eta(x_n)r_n^{2-\alpha}}{n}|\tilde v_{d,n}(x)|,
\end{align*}
whence
\begin{equation}\label{eq:termine g_d,n va a zero 2}
	\int_{B_\rho} G_{d,n}(x,\tilde v_{d,n}(x)) \, dx \to 0 \qquad \text{ as } n\to\infty.
\end{equation}
Notice now that any $J\subseteq [d]$ with $|J|=k$ is such that $J\cap \{1,\dots, d-k+1\}\neq \emptyset$. Since $\tilde v_{1,n}=\dots=\tilde v_{d-k+1,n} \equiv 0$ in $B_\rho$, $\tilde v_{d,n}=v_{d,n}$ in $B_R\setminus B_\rho$ and $\tilde v_{j,n} \le v_{j,n}$ in $B_R$ for $j=1,\dots,d-1$, the interaction terms are ordered:
\begin{equation}\label{diff prod}
\int_{B_R} M_n \sum_{\substack{J\subseteq [d] \\ |J|=k}} \gamma_{J} \tilde v_{J,n}^2\,dx \le \int_{B_R} M_n \sum_{\substack{J\subseteq [d] \\ |J|=k}} \gamma_{J} v_{J,n}^2\,dx.
\end{equation}
Moreover, since $v_{j,n}\to 0$ in $H^1(B_R(x_0))$ for $j=1,\dots, d-k+1$, also $\tilde v_{j,n}\to 0$, then
\begin{equation}\label{diff 3}
\int_{B_R} \left(|\nabla \tilde v_{j,n}|^2-|\nabla v_{j,n}|^2\right)\,dx \to 0 
\end{equation}
as $n \to \infty$. Therefore, by \eqref{diff 1}-\eqref{diff 3}, we obtain
\[
\begin{split}
c_n &\le J_{M_n,\bs g_n}(\tilde{\mf{v}}_n,B_R) \\
& \le J_{M_n,\bs g_n}(\mf{v}_n,B_R) + \sum_{j=1}^{d-k+1}\int_{B_R} \frac{1}{2}\left(|\nabla \tilde v_{j,n}|^2 -|\nabla v_{j,n}|^2\right)\,dx + \int_{B_\rho}\frac{1}{2} \left(|\nabla \tilde v_{d,n}|^2-|\nabla v_{d,n}|^2\right)\,dx \\
& + \sum_{j=1}^{d-k+1}\int_{B_R} \Big(G_{j,n}(x,v_{j,n}(x))-G_{j,n}(x,\tilde v_{j,n}(x))\Big)\,dx + \int_{B_\rho} \Big(G_{d,n}(x,v_{d,n}(x))-G_{d,n}(x,\tilde v_{d,n}(x))\Big)\,dx\\
& \le c_n -\delta + o(1),
\end{split}
\]
which is a contradiction for sufficiently large $n$.
\end{proof}

\begin{lemma}\label{lemma: 1 to k-2 divergent implies boundedness}
It is not possible that $1\leq |I_\infty|\leq k-2$ with
\[
M_n \prod_{i\in I_\infty} v_{i,n}^2(0) \le C.
\]
\end{lemma}

\begin{proof}
We argue by contradiction. If 
\begin{equation}\label{eq:17}
M_n \prod_{i\in I_\infty} v_{i,n}^2(0) \le C,
\end{equation}
then similarly to Lemma \ref{lemma: tra k-1 e d-1 diverg} we have that
for any $J\subset I_\infty$ with $|J|\leq |I_\infty|-1$,
\begin{equation*}
	M_n\prod_{i\in J} v_{i,n}^2(0) \to 0,
\end{equation*}
and hence, since $v_{i,n}$ has locally bounded oscillations and $v_{i,n}(0)\to+\infty$ for $i\in I_\infty$, up to subsequences:
\begin{enumerate}
	\item[$a)$] $\forall J\subseteq I_\infty$, $|J|=|I_\infty|$, $M_n v_{J,n}^2\to \lambda_J\geq 0 $ locally uniformly in $\R^N$;
	\item[$b)$] $\forall J\subseteq I_\infty$, $|J|\leq |I_\infty|-1$, $M_n v_{J,n}^2\to 0 $ locally uniformly in $\R^N$.
\end{enumerate}
Moreover $M_n\to 0$. Notice that for every $i\in I_F$ we have for almost every $x$
\begin{equation*}
	\Delta v_{i,n} = M_n v_{i,n} \sum_{\substack{J\subseteq I_\infty,\, J'\subseteq I_F\setminus\{i\} \\ |J|=|I_\infty| \\ |J'|=k-1-|I_\infty|}} \gamma_{J,J',i} v_{J,n}^2 v_{J',n}^2 + M_n v_{i,n} \sum_{\substack{J\subseteq I_\infty, J'\subseteq I_F\setminus\{i\} \\ |J|\leq |I_\infty|-1 \\ |J'|=k-1-|J| }} \gamma_{J,J',i}v_{J,n}^2  v_{J',n}^2-g_{i,n}(x,v_{i,n}).
\end{equation*}
Notice that, by $a)$, 
\begin{equation*}
	M_n v_{i,n} \sum_{\substack{J\subseteq I_\infty,\, J'\subseteq I_F \\ |J|=|I_\infty| \\ |J'|=k-1-|I_\infty|}} \gamma_{J,J',i} v_{J,n}^2 v_{J',n}^2 \to v_i\sum_{\substack{J\subseteq I_F\setminus\{i\} \\ |J|=k-1-|I_\infty|}} \lambda_{i,J} v_{J}^2, \quad \text{locally uniformly as $n\to\infty$},
\end{equation*}
where $\lambda_{i,J}\geq 0$ is symmetric in the sense of \eqref{eq:symmetry interaction coefficients}, while the second sum goes to $0$ locally uniformly by $b)$. Recalling \eqref{eq: g_n to 0}, by $H^1_{\loc}$ convergence we obtain that $v_{i}$ is a nonnegative solution of
\begin{equation}\label{eq:6}
	\Delta v_i = v_{i}\sum_{\substack{J\subseteq I_F\setminus \{i\}\\ |J|=k-|I_\infty|-1} } \lambda_{i,J} v_{J}^2 \quad \text{for every }i\in I_F.
\end{equation}
By Corollary \ref{cor: liouville intere interazione}, $v_{i}$ is constant for any $i\in I_F$, and this implies that $1\in I_\infty$. At this point, we aim at reaching a contradiction. Observe that
\begin{equation*}
	\Delta w_{1,n}= M_n v_{1,n} \sum_{\substack{J\subseteq I_\infty\setminus\{1\}, J'\subseteq I_F \\ |J|\leq |I_\infty|-1 \\ |J'|=k-1-|J| }} \gamma_{J,J',1}v_{J,n}^2  v_{J',n}^2-g_{1,n}(x,v_{1,n}).
\end{equation*}
Since for any $J\subseteq I_\infty\setminus\{1\}$ and for any $x\in B_r$ one has, by \eqref{eq:17}, 
\begin{equation*}
	M_n v_{1,n}v_{J,n}^2 \leq \frac{M_n v_{1,n}^2 v_{J,n}^2}{\inf_{B_{r}} v_{1,n}}\leq \frac{C}{\inf_{B_{r}} v_{1,n}}\to 0,
\end{equation*}
recalling also \eqref{eq: g_n to 0} we obtain $\Delta w_{1,n}\to 0$ in $L^{1}_{\loc}(\R^N)$, so by Lemma \ref{lemma: properties w_i} $w_1$ is harmonic, globally $\alpha$-H\"older continuous and non-constant, which is the desired contradiction.
\end{proof}

We are now ready to exclude the last case of diverging components.
\begin{lemma}\label{lemma: not 1 to k-2 divergent}
It is not possible that $1\leq |I_{\infty}|\leq k-2$.
\end{lemma}

\begin{proof}
We argue by contradiction. First notice that, thanks to Lemma \ref{lemma: 1 to k-2 divergent implies boundedness}, up to a subsequence
\begin{equation}
M_n \prod_{i\in I_\infty} v_{i,n}^2(0) \to +\infty.
\end{equation}

\medskip

\emph{Step 1)} For all $J\subseteq I_F$ with $|J|=k-|I_{\infty}|$ one has $\prod_{i\in J}v_{i}\equiv 0$.

\medskip

Let $J$ be such a set of indexes and suppose that for some $x_0\in \R^N$ and $i\in J$ one has 
$$\prod_{\substack{j\in J \\ j\neq i}} v_{j}(x_{0})>0.$$
If we show that $v_{i}(x_0)=0$, the claim is proved. 

By local uniform convergence we can find $\rho,\delta>0$ such that 
$$\prod_{\substack{j\in J\\ j\neq i}} v_{j,n}\geq \delta \quad \text{in $B_{\rho}(x_0)$}$$
 for every $n$ large enough. Moreover, since $\mf v_n$ has locally bounded oscillations, 
 \begin{equation*}
 	I_n:=\inf_{B_\rho(x_0)} M_n \prod_{i\in I_\infty} v_{i,n}^2 \to +\infty.
 \end{equation*}
 Then clearly
 \begin{equation*}
 	-\Delta v_{i,n} \leq -\delta^2 C I_n v_{i,n}+g_{i,n}(x,v_{i,n}(x)) \quad \text{in $B_\rho(x_0)$},
 \end{equation*}
 so by Lemma \ref{lemma:decadimento esponenziale}
 \begin{equation*}
 	\|v_{i,n}\|_{L^2_{B_{\rho/2}(x_0)}}\leq C e^{-C\sqrt{I_n}}+\frac{C}{I_n},
 \end{equation*}
 for some $C>0$ independent of $n$. Then $v_{i}(x_0)=0$.

\medskip

\emph{Step 2)} $1\in I_F$

\medskip

Notice that for every $j\in I_F$
\begin{equation*}
	\Delta v_{j,n} \geq -g_{j,n}(x,v_{j,n}),
\end{equation*}
so by \eqref{eq: g_n to 0} and $H^1_{\loc}$ convergence we have that $\Delta v_{j}\geq 0$ weakly. Then for any $J\subseteq I_F$ with $|J|=k-|I_\infty|$ we have that
\begin{equation*}
	\begin{cases}
		\Delta v_{j}\geq 0 & \forall j\in J \\
		v_{j}\geq 0 & \forall j\in J \\
		\prod_{j\in J} v_{j} \equiv 0
	\end{cases} 
	\quad \text{in $\R^N$}.
\end{equation*}
By Corollary \ref{cor: liou seg}, there exists $j\in J$ such that $v_{j}\equiv 0$.
Since any $k$-tuple $J'\subseteq \{1,\dots,d\}$ contains at least $k-|I_\infty|$ indexes in $I_F$, there exists $j'\in J'$ such that $j'\in I_F$ and $v_{j'}\equiv 0$. Then, by Lemma \ref{lemma: basic convergence when v(0) bounded} $(iii)$, for any $r>0$ and $J'\subseteq [d]$, $|J'|=k$, there holds
\begin{equation}\label{eq: termine interaz integrale va a 0}
	\int_{B_r} M_n v_{J',n}^2 \, dx\leq \|v_{j',n}\|_{L^\infty(B_r)}\int_{B_r} M_n v_{j',n} \prod_{\substack{i\in J' \\ i\neq j'}} v_{i,n}^2 \, dx \to 0 \quad \text{as $n\to\infty$}.
\end{equation}
As a consequence, for any $i\in I_\infty$, $r>0$,
\begin{equation*}
	\int_{B_r} |\Delta v_{i,n}|\, dx \leq \frac{M_n}{\inf_{B_r} v_{i,n}}  \int_{B_r}v_{i,n}^2\sum_{\substack{J\subseteq [d]\setminus \{i\} \\ |J|=k-1}}\gamma_{J,i}v_{J,n}^2 \, dx+\int_{B_r} |g_{i,n}(x,v_{i,n}(x))|\, dx \to 0,
\end{equation*}
recalling also \eqref{eq: g_n to 0}. Then $\Delta w_{1,n}\to 0$ in $L^1_{\loc}$ and by Lemma \ref{lemma: properties w_i} $w_1$ is harmonic in $\R^N$ and globally $\alpha$-H\"older continuous, thus constant. This rules out the case $1\in I_\infty$. Then, by Lemma \ref{lemma: basic convergence when v(0) bounded}, $v_{1,n}\to v_{1}$ as $n\to\infty$ in $H^1_{\loc}$ and locally uniformly.
\medskip

\emph{Step 3)} $\Delta v_{1}=0$ in $\{v_{1}>0\}$.

\medskip

Let $x_0 \in \R^N$ be such that $v_1(x_{0})>0$. Then, there exist $\delta, \rho>0$ such that $\inf_{B_\rho(x_0)}v_{1,n}\geq \delta$ for all $n$ big enough. Then, recalling \eqref{eq: g_n to 0}, we have
\begin{equation}\label{eq: stima laplaciano L^1 va a zero}
	\begin{split}
		\int_{B_{\rho}(x_0)} |\Delta v_{1,n}| \, dx &\leq \frac{M_n}{\inf_{B_\rho(x_0)}v_{1,n}}\int_{B_\rho}v_{1,n}^2\sum_{\substack{J\subseteq [d]\setminus \{1\} \\ |J|=k-1}}\gamma_{J,1}v_{J,n}^2\, dx+\int_{B_{\rho}(x_0)} |g_{1,n}(x,v_{1,n}(x))|\, dx \\
		&\leq \frac{M_n}{\delta}\int_{B_{\rho}(x_0)}v_{1,n}^2\sum_{\substack{J\subseteq [d]\setminus \{1\} \\ |J|=k-1}}\gamma_{J,1}v_{J,n}^2 \, dx + o(1).
	\end{split}
\end{equation}
By \eqref{eq: termine interaz integrale va a 0}, we infer that $\Delta v_{1,n} \to 0$ in $L^1(B_\rho (x_0))$, and with an analogous argument as in Lemma \ref{lemma: properties w_i} with $v_{1,n}$ instead of $w_{1,n}$ and restricted to $B_\rho (x_0)$ instead of the whole $\R^N$, one can prove that $v_{1}$ is harmonic in $B_\rho (x_0)$. By arbitrariness of $x_0$, Step 3 is proved.

\medskip

\emph{Step 4)} Final contradiction

\medskip

Since in Step 2 we proved that for any $J\subseteq I_F$ with $|J|=k-|I_\infty|$ there exists $j\in J$ such that $v_{j}\equiv 0$, there are at least $|I_F|-(k-|I_\infty)+1=d-k+1$ components whose limit is identically $0$. Moreover, $v_{1}$ is not constant, hence $i=1$ is not among such indexes. Therefore, to summarize:
\begin{itemize}
	\item $v_{1,n}\to v_{1}$ locally uniformly and in $H^1_{\loc}(\R^N)$ with $v_{1}\geq 0$ non constant;
	\item at least $d-k+1$ components tend to $0$ locally uniformly and in $H^1_{\loc}(\R^N)$.
\end{itemize}
Observe moreover that $v_{1}|_{S_{R}}\not\equiv 0$ for any $R\geq 1$: otherwise, since $v_{1}$ is subharmonic, by the maximum principle we would have $v_{1}\equiv 0$ in $B_1$, contradicting Lemma \ref{lemma: basic convergence when v(0) bounded} $(ii)$. Then we can apply Lemma \ref{lemma: minimality implies positivity} for arbitrarily large $R$ and obtain that $v_{1}>0$ in $\R^N$. Then, by Step 3, $v_1$ is harmonic and globally $\alpha$-H\"older continuous, thus constant, which gives the desired contradiction.
\end{proof}

So far, by Lemmas \ref{lemma: not d divergent}, \ref{lemma: tra k-1 e d-1 diverg} and \ref{lemma: not 1 to k-2 divergent}, we have proved that the sequence $\{\mf{v}_n(0)\}$ is bounded. To prove the contradiction that gives the uniform H\"older estimates, we need to rule out this last case. This is the aim of the next lemmas. Recall that at this stage the conclusions of Lemma \ref{lemma: basic convergence when v(0) bounded} apply to all the components.

\begin{lemma}\label{lemma: not v e M bdd}
It is not possible that both $\{\mf{v}_n(0)\}$ and $\{M_n\}$ are bounded.
\end{lemma}

\begin{proof}
Suppose by contradiction that this is the case, then up to a subsequence $M_n \to M \ge 0$ and $\mf{v}_n \to \mf{v}$ locally uniformly and in $ H^1_{\loc}(\R^N)$. Recalling also \eqref{eq: g_n to 0}, we obtain that the limit function $\mf v$ solves the system
\begin{equation}\label{ent sys 308}
\begin{cases}
\Delta v_i = M v_i \displaystyle\sum_{\substack{J\subseteq [d]\setminus \{i\} \\ |J|=k-1}} \gamma_{J,i} v_{J}^2& \text{in $\R^N$},\\
v_i \ge 0 & \text{in $\R^N$}.
\end{cases}
\end{equation}
By Corollary \ref{cor: liouville intere interazione}, $v_1$ is constant, contradicting Proposition \ref{thm: Liouville 1 comp}. 
\end{proof}

It remains to consider the case when $\{\mf{v}_n(0)\}$ is bounded and $M_n \to +\infty$ up to a subsequence. 

\begin{lemma}\label{lemma: not M unb e v bdd}
It is not possible that $\{\mf{v}_n(0)\}$ is bounded and $\{M_n\}$ is unbounded.
\end{lemma}
 \begin{proof}
By contradiction, assume $\{\mf{v}_n(0)\}$ is bounded and $M_n \to +\infty$ up to subsequences. By $H^1_{\loc}$ convergence and \eqref{eq: g_n to 0}, we have that the limits $v_i$ are subharmonic functions. Moreover, by Lemma \ref{lemma: basic convergence when v(0) bounded} $(iii)$ we deduce that 
$$ \prod_{i\in J} v_{i} \equiv 0 \quad \text{in $\R^N$} \qquad \text{for every $J\subseteq [d]$ with $|J|=k$.} $$ 

Then, by Corollary \ref{cor: liou seg}, we obtain that for every $J\subseteq [d]$ with $|J|=k$ there exists $j\in J$ such that $v_{j}\equiv 0$ in $\R^N$, hence there are at least $d-k+1$ components whose limit is identically $0$ (notice that, since $v_1$ is non-constant, $i=1$ is not among these indexes). With analogous estimates as \eqref{eq: termine interaz integrale va a 0} and \eqref{eq: stima laplaciano L^1 va a zero} in Lemma \ref{lemma: not 1 to k-2 divergent} one can prove that $v_1$ is harmonic on its positivity set.

Now, since $v_1|_{S_R}\not \equiv 0$ for all $R\geq 1$ and $d-k+1$ components vanish, Lemma \ref{lemma: minimality implies positivity} gives $v_1>0$ in $\R^N$ and hence $v_1$ is harmonic, globally $\alpha$-H\"older continuous and non-constant, contradicting Proposition \ref{thm: Liouville 1 comp}.
\end{proof}

\begin{proof}[Conclusion of the proof of Theorem \ref{thm: regolarità interna}]
Lemmas \ref{lemma: basic prop}-\ref{lemma: not M unb e v bdd} show the validity of \eqref{eq:Linf alpha estimate on balls}, and thus, by a covering argument, of \eqref{eq:interior estimate thm}. 

We now turn to part $(ii)$. Suppose that $\{\mf u_\beta\}_\b$ is uniformly bounded in $L^\infty(\Omega)$ and $\bs f_\b$ satisfies \eqref{eq:growth assumption strong}. Then, clearly, $\{\mf u_\b\}_\b$ is bounded in $C^{0,\alpha}_{\loc}(\Omega)$ for every $\alpha \in (0,\bar\nu)$. The Ascoli-Arzel\`a theorem now implies the local $C^{0,\alpha}$ convergence, up to subsequences, to some limit $\tilde{\mf{u}}\in C^{0,\alpha}_{\loc}(\Omega)$. 

The other properties can be now easily proved. By testing the equation of $u_{i,\beta}$ with a cut-off function $\varphi \in C^\infty_c(\Omega)$, we obtain
\begin{equation}\label{eq:bound termini inter}
\beta \int_{\Omega} u_{i,\beta} \varphi \sum_{\substack{J\subseteq [d]\setminus \{i\} \\ |J|=k-1}} \gamma_{J,i} u_{J,\b}^2 \,dx = \int_{\Omega} u_{i,\beta} \Delta \varphi \,dx + \int_{\Omega} f_{i,\b}(x,u_{i,\b}(x)) \varphi \, dx \le C.
\end{equation}
Thus, by local uniform convergence as $\beta \to +\infty$, the limit $\tilde{\mf{u}}$ satisfies the partial segregation condition

\begin{equation}\label{eq:partial segregation u tilde}
	\prod_{i\in J} \tilde u_{i} \equiv 0 \quad \text{in $\Omega$} \quad \text{for any $J\subseteq [d]$ with $|J|=k$}.
\end{equation}
To prove the strong $H^1_{\loc}$ convergence, we first test the equation of $u_{i,\beta}$ with $u_{i,\beta} \varphi^2$, where $\varphi \in C^\infty_c(\Omega)$ is any nonnegative cut-off function, obtaining
\begin{equation*}
	\int_{\Omega} \Big(|\nabla u_{i,\b}|^2 \varphi^2 +2u_{i,\b} \varphi \nabla u_{i,\b}\cdot \nabla \varphi +\b \varphi^2 \sum_{\substack{J\subseteq [d] \\ |J|=k}} \gamma_J u_{J,\b}^2 - f_{i,\b}(x,u_{i,\b}(x))u_{i,\b}\varphi^2 \Big) dx = 0,
\end{equation*} 
whence
\[
\int_{\Omega} |\nabla u_{i,\b}|^2 \varphi^2\,dx \le 4 \int_{\Omega}  u_{i,\b}^2 |\nabla \varphi|^2\,dx+\int_{\Omega} f_{i,\b}(x,u_{i,\b})u_{i,\b}\varphi^2 \, dx.
\]
This shows the boundedness of $\{u_{i,\b}\}$ in $H^1_{\loc}(\Omega)$, and hence up to a subsequence $u_{i,\b} \weak \tilde u_i$ weakly in $H^1_{\loc}(\Omega)$ as $\b\to+\infty$. 
Now, testing the equation of $u_{i,\beta}$ with $(u_{i,\beta}-\tilde u_i) \varphi$ and integrating by parts, we obtain
\begin{multline*}
	\left|\int_{\Omega} \varphi \nabla u_{i,\b} \cdot \nabla (u_{i,\b}-\tilde u_{i})  \, dx \right|\\
	\le \|u_{i,\b}-\tilde u_{i}\|_{L^\infty(\mathrm{supp}(\varphi))} \int_{\Omega} \Big( |\nabla u_{i,\b} \cdot \nabla \varphi| + \b \varphi \, u_{i,\b} \sum_{\substack{J\subseteq [d]\setminus \{i\} \\ |J|=k-1}} \gamma_{J,i} u_{J,\b}^2+|f_{i,\b}(x,u_{i,\b})|\varphi \Big)dx,
\end{multline*}
and the right hand side goes to zero by local uniform convergence and \eqref{eq:bound termini inter}, proving strong convergence in $H^1_{\loc}(\Omega)$. 

Now notice that, by \eqref{eq:partial segregation u tilde}, for any $J\subseteq [d]$, $|J|=k$, we have
\begin{equation*}
	\bigcup_{j\in J} \{\tilde u_j=0\} = \Omega.
\end{equation*}
Then, for any compact $K\subset \Omega$ and any such $J$,
\[
\beta \int_{K} \gamma_{J} \prod_{j\in J} u_{j,\beta}^2\,dx \leq \b \gamma_{J} \sum_{j\in J} \|u_{j,\b}\|_{L^\infty(K\cap \{\tilde u_{j}=0\})} \int_{K\cap \{\tilde u_{j}=0\}} u_{j,\b} \prod_{j'\in J\setminus \{j\}} u_{j',\b}^2 \to 0,
\]
and this by \eqref{eq:bound termini inter} and local uniform convergence concludes the proof.
\end{proof}

\section{Global H\"older bounds}\label{sec: thm regolarità bordo}

The aim of this section is to prove Theorem \ref{thm: regolarità al bordo}. In the first subsection, we establish the global H\"older estimate \eqref{eq:global estimate} through a blow-up analysis, with particular attention to sequences concentrating near the boundary. In the second subsection, under the uniform growth assumption, we derive uniform $L^\infty$ bounds from the assumed $H^1$ bounds by combining a Brezis--Kato-type iteration with a Stampacchia argument, and then conclude the proof of the theorem.

\subsection{Proof of \eqref{eq:global estimate}}
The proof of the global estimate is based on a blow-up analysis around pairs of points where the H\"older quotient becomes large. The main additional feature with respect to the interior estimate is that the concentration scale may be comparable with the distance from the boundary, so that the limiting geometry of the rescaled domains has to be taken into account.

Fix $\alpha\in(0,\bar\nu)$. We argue by contradiction and suppose that, for some boundary datum $\bs\psi$ and some family $\{\bs f_\b\}_{\b>0}$ satisfying the assumptions of the theorem, there exist $\{\b_n\}_n\subset (0,+\infty)$ and corresponding nonnegative minimizers
\[
\mf u_n:=\mf u_{\b_n}
\]
of $J_{\b_n,\bs f_{\b_n}}(\cdot,\Omega)$ on $\mathcal U_{\bs\psi}$ such that
\[
\frac{
	[\mf u_n]_{C^{0,\alpha}(\overline\Omega)}
}{
	\|\mf u_n\|_{L^\infty(\Omega)}+R_n
}
>n
\qquad\text{for every }n\in\N,
\]
where, as in the proof of Theorem \ref{thm: regolarità interna}, we set
\[
R_n
:=
\sup\left\{
|f_{i,\b_n}(x,u)|:
x\in\Omega,\ 
u\in[0,\sup_\Omega u_{i,\b_n}],\
i=1,\dots,d
\right\}.
\]
In particular, defining
\[
L_n
:=
[\mf u_n]_{C^{0,\alpha}(\overline\Omega)}
=
\sup_{i=1,\dots,d}
\sup_{\substack{x\neq y\\x,y\in\overline\Omega}}
\frac{|u_{i,n}(x)-u_{i,n}(y)|}{|x-y|^\alpha},
\]
we have
\begin{equation}\label{eq:L_n global}
	L_n
	>
	n\left(
	\|\mf u_n\|_{L^\infty(\Omega)}+R_n
	\right).
\end{equation}

For every $n$, choose a component and a pair of points realizing the above H\"older quotient. After relabelling the components, we may write
\[
L_n
=
\frac{|u_{1,n}(x_n)-u_{1,n}(y_n)|}{|x_n-y_n|^\alpha}
\]
for some $x_n,y_n\in\overline\Omega$, $x_n\neq y_n$. Setting
\[
r_n:=|x_n-y_n|,
\]
estimate \eqref{eq:L_n global} gives
\[
r_n^\alpha
\leq
\frac{2\|\mf u_n\|_{L^\infty(\Omega)}}{L_n}
<\frac{2}{n},
\]
and hence $r_n\to0$.

We rescale the sequence around $x_n$ at this vanishing scale by setting
\[
v_{i,n}(x)
:=
\frac{u_{i,n}(x_n+r_nx)}{L_nr_n^\alpha},
\qquad
x\in\Omega_n
:=
\frac{\Omega-x_n}{r_n}.
\]
The subsequent analysis is determined by the asymptotic position of $x_n$ relative to the boundary at the scale $r_n$, namely by the behaviour of
\[
\frac{\dist(x_n,\partial\Omega)}{r_n}.
\]

If this quantity tends to $+\infty$, the boundary disappears in the blow-up limit and the domains $\Omega_n$ exhaust $\R^N$. In this case the argument reduces to the interior analysis developed in Section \ref{sec: thm regolarità interna}. Indeed, the original functions satisfy the local minimality property \eqref{eq:local minimality} on every $\Omega'\ssubset\Omega$, while in the present rescaling the sequence $\mf v_n$ simultaneously satisfies the normalized H\"older bound and the rescaled system. We can therefore apply the contradiction argument of the interior case.

It remains to consider the case in which $\dist(x_n,\partial\Omega)/r_n$ stays bounded. Since $\Omega$ is of class $C^1$, after passing to a subsequence the rescaled domains $\Omega_n$ converge to a half-space $\Omega_\infty$, with $0\in\overline{\Omega_\infty}$. The rest of this subsection is devoted to excluding this second possibility. We begin by collecting the basic properties of the rescaled sequence, which are the boundary counterparts of those established in Lemma \ref{lemma: basic prop} and that can be proved analogously.

\begin{lemma}\label{lemma: basic prop global}
	Under the previous notation,
	\begin{itemize}
		\item[($i$)] the sequence $\{{\mf{v}}_n\}$ has uniformly bounded $\alpha$-H\"older semi-norm in $\Omega_n$, and in particular
		\[
		\sup_{i=1,\dots,d} \sup_{\substack{x \neq y \\ x,y \in \Omega_n}}    \frac{ | v_{i,n}(x)- v_{i,n}(y)|}{|x-y|^{\alpha}} = \frac{ | v_{1,n}(0)- v_{1,n}\left(\frac{y_n-x_n}{r_n}\right)|}{\left|\frac{y_n-x_n}{r_n}\right|^{\alpha}} = 1
		\]
		for every $n$.
		\item[($ii$)] $v_{i,n}$ is a weak solution of 
		\begin{equation}\label{eq: eq satisfied by v_n global}
			-\Delta v_{i,n} = -M_n v_{i,n} \sum_{\substack{J\subseteq [d]\setminus \{i\}\\ |J|=k-1}} \gamma_{J,i} v_{J,n}^{2}+g_{i,n}(x,v_{i,n}), \quad v_{i,n}\geq 0 \qquad \text{in $\Omega_n$},
		\end{equation}
		where
		\[
		M_n:= \beta_n r_n^{2+2(k-1)\alpha} L_n^{2(k-1)} \quad \text{ and } \quad g_{i,n}(x,u):=\frac{r_n^{2-\alpha}}{L_n}f_{i,\b_n}(x_n+r_n x,r_n^\alpha L_n u).
		\] 
		Furthermore, 
		\begin{equation}\label{eq: g_n to 0 global}
			g_{i,n}(x,v_{i,n}(x))\to 0 \quad \text{locally uniformly in $\Omega_\infty$ as $n\to\infty$ for every $i=1,\dots,d$.}
		\end{equation}
		
		\item[($iii$)] $\mf{v}_n$ is a minimizer of \eqref{eq: eq satisfied by v_n} with respect to variations with compact support, namely for every $\Omega' \ssubset \Omega_\infty$
		\[
		J_{M_n,\bs g_n}(\mf{v}_n,\Omega') \le J_{M_n,\bs g_n}(\mf{v}_n+ \bs{\varphi},\Omega') \quad \forall \bs{\varphi} \in H^1_0(\Omega', \R^d),
		\]
		for sufficiently large $n$, where $J_{M_n,\bs g_n}(\cdot, \Omega)$ is defined as in \eqref{eq:def J_beta}.
	\end{itemize}
\end{lemma} 

An analogous of Lemma \ref{lemma: basic convergence when v(0) bounded} holds when a component is bounded at the origin, with the same proof.

\begin{lemma}\label{lemma: basic convergence when v(0) bounded global}
	Suppose that $\{v_{i,n}(0)\}$ is bounded. Then there exists a function $v_i \in C^{0}(\Omega_\infty) \cap H^1_{\loc}(\Omega_\infty)$, globally $\alpha$-H\"older continuous, such that, up to a subsequence: 
	\begin{itemize}
		\item[($i$)] $v_{i,n} \to v_i$
		locally uniformly on $\Omega_\infty$; 
		\item[($ii$)] $v_1$ is non-constant in $B_1\cap \Omega_{\infty}$;
		\item[($iii$)] for every $K \subset \Omega_\infty$ compact there exists $C>0$ such that
		\[
		M_n \int_{K} \Bigg(\sum_{\substack{J\subseteq [d]\setminus \{i\} \\ |J|=k-1}}\gamma_{J,i} v_{J,n}^2 \Bigg) v_{i,n}\,dx \le C
		\]
		\item[($iv$)] $v_{i,n} \to v_i$ strongly in $H^1_{\loc}(\Omega_{\infty})$.
	\end{itemize}
\end{lemma}

We now recall a useful result that gives information about the behaviour of the limit functions at the boundary of the half-space (see \cite[Lemma 4.3]{SoaTer22}).

\begin{lemma}\label{lemma: limit on boundary}
It is possible to extend $\mf{v}_n$ outside $\Omega_n$ in a Lipschitz fashion, in such a way that:
\begin{itemize}
\item[($i$)] If $\{v_{i,n}(0)\}$ is bounded, then 
$v_{i,n} \to v_i$ in $C^{0,\alpha'}_{\loc}(\R^N)$ for every $0<\alpha'<\alpha$, up to a subsequence; moreover, the limit function $v_i$ attains a constant value on the boundary $\pa \Omega_\infty$.
\item[($ii$)] If $\{v_{i,n}(0)\}$ is unbounded, then $w_{i,n}(x) := v_{i,n}(x)-v_{i,n}(0)$ converges to $w_i$ in $C^{0,\alpha'}_{\loc}(\R^N)$ for every $0<\alpha'<\alpha$, up to a subsequence; moreover, the limit function $w_i$ attains a constant value on the boundary $\pa \Omega_\infty$. Moreover, $w_1$ is non-constant in $B_1\cap \Omega_\infty$.
\end{itemize}
\end{lemma}

We now aim at reaching a contradiction in any possible behaviour of $\{v_{i,n}(0)\}$ and $\{M_n\}$. The first preliminary observation is crucial:

\begin{lemma}\label{lemma: at least d-k+1 vanishing global}
	In this setting, we have that $|I_F|\geq d-k+1$ and at least $d-k+1$ components $v_{i}$ are such that $v_{i}=0$ in $\Omega_{\infty}$.
\end{lemma}
\begin{proof}
	Fix $R$ big enough such that $B_R(0)\cap \partial\Omega_\infty \neq \emptyset$, and let $z_n\in B_R(0)\cap \partial\Omega_n $ for every $n$ large. By the partial segregation condition \eqref{PSC}, clearly satisfied also by the rescaled functions $\psi_{i,n}$, at least $d-k+1$ components of $\mf v_n$ vanish at $x=z_n$. Since the set of indexes is finite, we can assume up to subsequences that such components do not depend on $n$, so for some $i_1,\dots, i_{d-k+1} \in \{1,\dots,d\}$ we have
	$$ v_{i_j,n}(z_n)=\psi_{i_j,n}(z_n)=0 \qquad \text{for every $n$ and $j=1,\dots,d-k+1.$} $$
	By uniform boundedness of the $\alpha$-H\"older seminorm, it follows that
	$$ v_{i_j,n}(0)\leq R^\alpha \qquad \text{for every $n$ and $j=1,\dots,d-k+1,$}  $$
	and this concludes the first part of the statement.
	Observe now that, for such indexes $i_1,\dots, i_{d-k+1}$, since $z_n\in B_R(0)\cap \partial\Omega_n$ local uniform convergence implies that, up to subsequences,
	$$ v_{i_j}(\bar x)=0 \qquad \text{for $j=1,\dots,d-k+1$}$$
	for some $\bar x\in \partial\Omega_\infty$. Since, by Lemma \ref{lemma: limit on boundary}, $v_{i_j}$ is constant on $\partial\Omega_{\infty}$, we conclude that $$v_{i_j}=0 \qquad \text{on $\partial\Omega_{\infty}$} \quad \text{for $j=1,\dots,d-k+1$}. $$
	Now, by Lemmas \ref{lemma: basic prop global} and \ref{lemma: basic convergence when v(0) bounded global}, we know that $v_i$ is subharmonic in $\Omega_\infty$ and $\alpha$-H\"older continuous up to the boundary.
	
	Extend $v_{i_j}$ by continuity as the zero function outside $\Omega_\infty$: we claim that such extension, that we still denote by $v_{i_j}$, is subharmonic on $\R^N$. Indeed, it is clearly subharmonic in $\Omega_\infty$ and in $\R^N\setminus\overline{\Omega_\infty}$, and for any $x_0\in \partial\Omega_{\infty}$ and $r>0$, by non-negativity of $v_{i_j}$,
	$$ 0=v_{i_j}(x_0)\leq \fint_{\partial B_r(x_0)} v_{i_j} \, d\sigma,$$
	thus $v_{i_j}$ satisfies a sub-mean value property, which is equivalent to subharmonicity in the sense of distributions. Since $v_{i,j}$ is also continuous in $\R^N$, it is in $H^1_{\loc}(\R^N)$ by \cite[Exercise 2.3]{Petrosyan_Shahgholian_Uraltseva_2012} and hence subharmonic in the weak sense.
	Recalling that $\partial\Omega_{\infty}$ is a hyperplane, consider $v'$ as the even reflection of $v_{i_j}$ across $\pa \Omega_\infty$.
	Then $v_{i_j}$ and $v'$ are subharmonic globally $\alpha$-H\"older continuous functions in $H^1_{\loc}(\R^N)$ such that $v_{i_j}v'\equiv 0$ in $\R^N$. Therefore, by Proposition \ref{prop: Liou sys 2}, either $v_{i_j}\equiv 0$ or $v'\equiv 0$. But then, by construction of $v'$, we conclude that $v_{i_j}\equiv 0$, and the proof is concluded.
\end{proof}

Thanks to the last lemma, we are now left to show a contradiction in the case $|I_\infty|\leq k-1$.

First we need the following Liouville-type theorem for harmonic functions on half-spaces, that is proved in a more general anisotropic setting in \cite[Lemma 4.2]{SoaTer22}.

\begin{lemma}\label{lem: liou half}
Let $H$ be a half-space, and suppose that $v \in C^0(\overline{H}) \cap C^2(H)$ is a harmonic function in $H$, with $v|_{\pa H} = const.$, and $v$ globally $\alpha$-H\"older continuous for some $\alpha \in (0,1)$. Then $v$ is constant in $H$.
\end{lemma}

We now rule out the case of diverging components.
\begin{lemma}\label{lemma: no diverging comp global}
It is not possible that $1\leq |I_\infty|\leq k-1$.
\end{lemma}

\begin{proof}
We argue by contradiction.
Fix $i\in I_F$. Since by Lemma \ref{lemma: at least d-k+1 vanishing global} for any $k$-tuple $J\subseteq [d]$ there exists $j\in J\cap I_F$ such that $v_{j}=0$, with estimates analogous to \eqref{eq: termine interaz integrale va a 0} and \eqref{eq: stima laplaciano L^1 va a zero} in Lemma \ref{lemma: 1 to k-2 divergent implies boundedness} we can prove that $v_i$ is harmonic in $\{v_i>0\}$.
Then, $v_{i}$ is weakly subharmonic in $\Omega_{\infty}$ and $v_i=c$ on $\pa\Omega_\infty$ for some $c\geq 0$ by Lemma \ref{lemma: limit on boundary}.
As a consequence,
\begin{itemize}
	\item if $c=0$, then arguing as in Lemma \ref{lemma: at least d-k+1 vanishing global} with the trivial extension outside $\Omega_{\infty}$ we obtain that $v_i\equiv 0$ in $\overline{\Omega_\infty}$;
	\item if $c>0$, then by uniform H\"older continuity there exist $\rho>0$ and $\delta>0$ such that $v_i(x) \geq \delta$ if $\dist(x,\pa\Omega_\infty) \leq \rho$. We apply Lemma \ref{lemma: minimality implies positivity} (which clearly still holds provided that $B_R(x_0)\ssubset\Omega_\infty$) on arbitrarily large balls $B_R\ssubset \Omega_{\infty}$ such that $\dist(B_R,\pa\Omega_\infty)<\rho$, so that $v_i|_{S_R}\not \equiv 0$, obtaining that $v_i>0$ in $\Omega_\infty$. Then $v_i$ is harmonic on $\Omega_{\infty}$, globally $\alpha$-H\"older continuous and constant on $\pa\Omega_{\infty}$, so $v_i\equiv c$ in $\overline{\Omega_{\infty}}$ by Lemma \ref{lem: liou half}.
\end{itemize}
Since $v_1$ is non-constant by Lemma \ref{lemma: basic convergence when v(0) bounded global}, this proves that $1\in I_\infty$.
Now, for every $B_r \ssubset \Omega_{\infty}$,
\begin{equation*}
	\int_{B_r} |\Delta w_{1,n}|\, dx \leq \frac{ M_n }{\inf_{B_r} v_{1,n}}\int_{B_r}v_{1,n}^2\sum_{\substack{J\subseteq [d]\setminus \{1\} \\ |J|=k-1}}\gamma_{J,1}v_{J,n}^2 \, dx+\int_{B_r} |g_{1,n}(x,v_{1,n}(x))|\, dx \to 0,
\end{equation*}
recalling also \eqref{eq: g_n to 0 global} and the estimate analogous to \eqref{eq: termine interaz integrale va a 0} that holds in this context. Then, noticing that the conclusion of Lemma \ref{lemma: properties w_i} $(ii)$ holds also in this context replacing $\R^N$ by $\Omega_\infty$, we obtain that $w_{1,n}\to w_1$ locally uniformly and in $H^1_{\loc}(\Omega_{\infty})$, with $w_1$ harmonic in $\Omega_\infty$, constant on $\pa\Omega_\infty$ and non-constant on $\Omega_\infty$ by Lemma \ref{lemma: limit on boundary}. But then Lemma \ref{lem: liou half} implies that $w_1$ is constant on $\Omega_\infty$, giving the desired contradiction.
\end{proof}

Finally, we rule out the case of all converging components.

\begin{lemma}\label{lemma: not bounded global}
	It is not possible that $\{\mf v_{n}(0)\}$ is bounded.
\end{lemma}
\begin{proof}
	Suppose by contradiction that $\{\mf v_n(0)\}$ is bounded. Since we know that at least $d-k+1$ components are identically zero by Lemma \ref{lemma: at least d-k+1 vanishing global}, with analogous estimates as \eqref{eq: termine interaz integrale va a 0} and \eqref{eq: stima laplaciano L^1 va a zero} in Lemma \ref{lemma: not 1 to k-2 divergent} we can prove that $v_1$ is harmonic in $\{v_1>0\}\cap \Omega_\infty$ (notice that this argument works either if $\{M_n\}$ is bounded or not). Then, since $v_1$ is not constant on $B_1\cap\Omega_\infty$ by Lemma \ref{lemma: basic convergence when v(0) bounded global}, there exists $x_0\in \Omega_\infty\cap B_1$ such that $v_1(x_0)>0$. Then, taking arbitrarily large balls $B_R\ssubset\Omega_\infty$ containing $x_0$, by subharmonicity of $v_1$ and Lemma \ref{lemma: minimality implies positivity} we conclude that $v_1>0$ in $\Omega_\infty$ and thus it is harmonic on this set. Since it is non-constant, but constant on $\pa\Omega_\infty$ by Lemma \ref{lemma: limit on boundary}, we reach a contradiction with Lemma \ref{lem: liou half}.
\end{proof}

Lemmas \ref{lemma: at least d-k+1 vanishing global} - \ref{lemma: not bounded global} show a contradiction for every possible behaviour of the blow-up sequence, thus proving \eqref{eq:global estimate}.

\subsection{Conclusion of the proof of Theorem \ref{thm: regolarità al bordo}}

We are left to prove the second part of the theorem. From now on, we assume that \eqref{eq:growth assumption strong} holds with $p<2^*$. 

We will make use of the following classical lemma due to Stampacchia; see \cite[Lemma B.1]{Kinderlehrer_Stampacchia_2000}.

\begin{lemma}\label{lemma:stampacchia}
	Let $\varphi:[k_0,+\infty)\to[0,+\infty)$ be a nonnegative and nonincreasing function such that
	\begin{equation*}
		\varphi(h)
		\leq
		\frac{C}{(h-k)^\alpha}\varphi(k)^\gamma,
		\qquad \forall h>k\geq k_0,
	\end{equation*}
	where $C,\alpha,\gamma$ are positive constants with $\gamma>1$. Then
	\begin{equation*}
		\varphi(k_0+d)=0,
	\end{equation*}
	where
	\begin{equation*}
		d^\alpha
		=
		C\varphi(k_0)^{\gamma-1}
		2^{\alpha\gamma/(\gamma-1)}.
	\end{equation*}
\end{lemma}

We first establish that uniform $H^1$ bounds yield uniform $L^\infty$ bounds. The proof combines a Brezis--Kato-type iteration with the Stampacchia lemma above, and it is at this stage that the strict inequality $p<2^*$ plays a crucial role.

\begin{lemma}\label{lemma:H^1 to Linfty}
	Under the hypotheses of Theorem \ref{thm: regolarità al bordo} $(ii)$, the family $\{\mf u_\b\}_{\b>0}$ is uniformly bounded in $L^\infty(\Omega)$.
\end{lemma}
\begin{proof}
	Fix $M>\|\bs \psi\|_{L^\infty(\Omega)}$ and let
	\[
	w_{i,\beta}:=(u_{i,\beta}-M)^+\in H_0^1(\Omega).
	\]
	Then, since 
	$$
	-\Delta(u_{i,\beta}-M)
	\leq C(1+|u_{i,\beta}|^{p-1}),
	$$
	it is easy to see (arguing as in \cite[Exercise 2.1]{Petrosyan_Shahgholian_Uraltseva_2012}) that, in the weak sense,
	\begin{equation}\label{eq:14}
	-\Delta w_{i,\beta}
	\leq
	C\bigl(1+|u_{i,\beta}|^{p-1}\bigr)
	\leq
	C'\bigl(1+|w_{i,\beta}|^{p-1}\bigr)
	=
	a_{i,\beta}(x)(1+|w_{i,\beta}|),
	\end{equation}
	
	where
	\[
	a_{i,\beta}(x)
	:=
	C'\frac{1+|w_{i,\beta}|^{p-1}}{1+|w_{i,\beta}|}.
	\]
	Since $p<2^*$, there exists $\delta,C>0$ such that
	\begin{equation}\label{eq:15}
	\|a_{i,\beta}\|_{L^{N/2+\delta}(\Omega)}
	\leq C
	\qquad\text{for every }i,\beta,
	\end{equation}
	by the uniform $H^1$ boundedness of $\{\mf u_{\b}\}_\b$.
	
	\medskip
	
	\textbf{Step 1}: for every $1\leq q<\infty$ there exists $C=C(q)>0$ independent of $\b$ such that
	$$
	\sup_{i,\b} \norm{w_{i,\b}}_{L^q(\Omega)} \leq C.
	$$
	
	\medskip
	
	If $N=2$, the claim simply follows by the Sobolev embedding and the uniform $H^1$ bounds. Therefore, suppose that $N\geq 3$.
	
	The idea is to carry out a Brezis-Kato type argument. For every $i,\b$ and $s>0, L>1$, set
	\[
	v_{i,\b}:=\min\{(w_{i,\beta})^s,L\}
	\]
	and 
	\[
	\varphi_{i,\b}=w_{i,\beta}v_{i,\b}^2.
	\]
	Since $\varphi_{i,\b}$ is a nonnegative function and belongs to $H_0^1(\Omega)$, we can use it as test function in \eqref{eq:14}, getting
	\[
	\int_\Omega \nabla w_{i,\beta}\cdot\nabla\varphi_{i,\b}
	\leq
	\int_\Omega a_{i,\beta}(1+w_{i,\beta})\varphi_{i,\b}.
	\]
	Since $\nabla v_{i,\b}=s w_{i,\b}^{s-1}\nabla w_{i,\b} \chi_{\{w_{i,\b}^s \leq L\}}$, we have
	\[
	\begin{aligned}
		\int_\Omega \nabla w_{i,\beta}\cdot\nabla\varphi_{i,\beta}
		&=
		\int_\Omega v_{i,\beta}^2|\nabla w_{i,\beta}|^2
		+
		2s\int_{\{w_{i,\beta}^s\leq L\}}
		\,w_{i,\beta}^{2s}|\nabla w_{i,\beta}|^2
		\\
		&\geq
		\int_\Omega v_{i,\beta}^2|\nabla w_{i,\beta}|^2.
	\end{aligned}
	\]
	Notice also that
	\[
	\begin{aligned}
		\int_\Omega
		\left|\nabla(w_{i,\beta}v_{i,\b})\right|^2
		&=
		\int_\Omega
		\left|v_{i,\beta}\nabla w_{i,\beta}
		+w_{i,\beta}\nabla v_{i,\beta}\right|^2
		\\
		&\leq
		2\int_\Omega v_{i,\beta}^2|\nabla w_{i,\beta}|^2
		+
		2s^2
		\int_{\{w_{i,\beta}^s\leq L\}}
		w_{i,\beta}^{2s}|\nabla w_{i,\beta}|^2
		\\
		&\leq
		C\int_\Omega
		v_{i,\beta}^2|\nabla w_{i,\beta}|^2,
	\end{aligned}
	\]
	where now and in the following $C$ denotes a positive constant depending on $s$ and not on $\b$, and that may vary from a line to another.
	
	Hence, combining the last inequalities and noticing that $w_{i,\b}v_{i,\b}^2+w_{i,\b}^2v_{i,\b}^2\leq 2 + 2w_{i,\b}^2v_{i,\b}^2$,
	\[
	\begin{aligned}
		\int_\Omega
		\left|\nabla(w_{i,\beta}v_{i,\beta})\right|^2
		&\leq
		C\int_\Omega
		a_{i,\beta}(1+w_{i,\beta})w_{i,\beta}v_{i,\beta}^2
		\\
		&\leq
		C\big(1+
		2K\int_{\{a_{i,\beta}\leq K\}}
		w_{i,\beta}^2v_{i,\beta}^2
		+
		\int_{\{a_{i,\beta}>K\}}
		2a_{i,\beta} w_{i,\beta}^2v_{i,\beta}^2\big)
		\\
		&\leq
		C\left(1
		+
		2K\|w_{i,\beta}\|_{L^{2s+2}(\Omega)}^{2s+2}
		+
		2\left(
		\int_{\{a_{i,\beta}>K\}}
		a_{i,\beta}^{N/2}
		\right)^{2/N}
		\|w_{i,\beta}v_{i,\beta}\|_{L^{2^*}(\Omega)}^2\right)
		\\
	\end{aligned}
	\]
	for every $K>0$ and where in the last inequality we used that $v_{i,\b}\leq w_{i,\b}^s$.
	
	Notice that, by \eqref{eq:15}, we have
	$$
	\int_{\{a_{i,\beta}>K\}}
	a_{i,\beta}^{N/2} \leq \frac{1}{K^{\delta}} \int_{\{a_{i,\beta}>K\}} a_{i,\beta}^{N/2+\delta} \leq \frac{C}{K^{\delta}}.
	$$
	
	Therefore, combining the last two estimates,
	$$
	\int_\Omega
	\left|\nabla(w_{i,\beta}v_{i,\beta})\right|^2
	\leq
	C(1
	+
	2K\|w_{i,\beta}\|_{L^{2s+2}(\Omega)}^{2s+2}
	+
	\varepsilon(K)
	\|\nabla(w_{i,\beta}v_{i,\beta})\|_{L^2(\Omega)}^2),
	$$
	where $\varepsilon(K)\to0$ as $K\to+\infty$, uniformly with
	respect to $\beta$.
	
	Thus, for every fixed $s>0$, choosing $K$ sufficiently large, we obtain

	\[
	\|\nabla(w_{i,\beta}v_{i,\beta})\|_{L^2(\Omega)}^2
	\leq
	C
	+
	C\|w_{i,\beta}\|_{L^{2s+2}(\Omega)}^{2s+2}.
	\]
	Letting $L\to+\infty$, Fatou's Lemma gives
	\begin{equation}\label{eq:16}
	\|\nabla(w_{i,\beta}^{s+1})\|_{L^2(\Omega)}^2
	\leq
	C
	+ C\|w_{i,\beta}\|_{L^{2s+2}(\Omega)}^{2s+2}.
	\end{equation}
	Hence, if
	\[
	\sup_{i,\b}\|w_{i,\beta}\|_{L^{2s+2}(\Omega)}
	\leq C
	\]
	for some $C>0$,	by \eqref{eq:16} we deduce that
	\[
	w_{i,\beta}^{s+1}\in H_0^1(\Omega),
	\qquad
	\sup_{i,\b}\|w_{i,\beta}^{s+1}\|_{H^1_0(\Omega)}
	\leq C',
	\]
	where $C'>0$ depends on $s$.
	Therefore, by Sobolev's inequality,
	\[
	w_{i,\beta}\in L^{2^*(s+1)}(\Omega), \qquad
	\sup_{i,\b}\|w_{i,\beta}\|_{L^{2^*(s+1)}(\Omega)}
	\leq C,
	\]
	and observe that
	$$
	2^*(s+1)=2s'+2 \quad \iff \quad s'= \frac{N}{N-2}s+\frac{2}{N-2}.
	$$
	Starting from $s_0=\frac{2}{N-2}$, we can therefore iterate the reasoning, obtaining a diverging sequence of exponents $2s+2$. This concludes the proof of Step 1. 
	
	\medskip
	
	\textbf{Step 2}: from uniform $L^q$ to uniform $L^\infty$ estimate
	
	\medskip
	
	Suppose first that $N\geq 3$. Pick
	\[
	q>\frac{N}{2}(p-1),
	\]
	so that
	\[
	r=\frac{q}{p-1}>\frac N2.
	\]
	Recalling \eqref{eq:14} and Step 1, we have that, in the weak sense,
	\[
	-\Delta w_{i,\beta}
	\leq
	g_{i,\b},
	\]
	with
	\[
	\sup_{i,\beta}
	\norm{g_{i,\b}}_{L^r(\Omega)}\leq C
	\]
	for some $C>0$.
	
	For every $k\geq 0$ and $i,\b$ fixed, let
	\[
	A_{\beta,k}
	:=
	\{x\in\Omega:w_{i,\beta}(x)>k\},
	\qquad
	v_{\beta,k}:=(w_{i,\beta}-k)^+,
	\]
	where we drop for convenience the index $i$.
	Clearly $v_{\beta,k}\in H_0^1(\Omega)$, and hence, applying H\"older's inequality with $r$ and $r'$ and noting that $r'<2^*$,
	\[
	\begin{aligned}
		\int_\Omega |\nabla v_{\beta,k}|^2=\int_\Omega \nabla w_{i,\b} \cdot \nabla v_{\b,k}
		&\leq
		\int_\Omega
		g_{i,\b} v_{\beta,k}
		\\
		&\leq
		C\,|A_{\beta,k}|^{\frac1{r'}-\frac1{2^*}}
		\|v_{\beta,k}\|_{L^{2^*}(\Omega)},
	\end{aligned}
	\]
	with $C$ independent of $k,i$ and $\b$.
	By the Sobolev inequality,
	\[
	\norm{v_{\beta,k}}^2_{L^{2^*}(\Omega)}
	\leq
	C\,|A_{\beta,k}|^{\frac1{r'}-\frac1{2^*}}
	\| v_{\beta,k}\|_{L^{2^*}(\Omega)},
	\]
	and therefore
	\[
	\norm{v_{\beta,k}}_{L^{2^*}(\Omega)}
	\leq
	C\,|A_{\beta,k}|^{\frac1{r'}-\frac1{2^*}},
	\]
	where the inequality is trivially true also when
	$v_{\beta,k}=0$.
	
	On the other hand, if
	$k<h$,
	\[
	v_{\beta,k}
	=
	w_{i,\beta}-k
	\geq h-k
	\qquad\text{on }A_{\beta,h}.
	\]
	Thus
	\[
	\|v_{\beta,k}\|_{L^{2^*}(\Omega)}
	\geq
	|A_{\beta,h}|^{1/2^*}(h-k),
	\]
	and consequently
	\[
	(h-k)|A_{\beta,h}|^{1/2^*}
	\leq
	C\,|A_{\beta,k}|^{\frac1{r'}-\frac1{2^*}}.
	\]
	Therefore
	\[
	|A_{\beta,h}|
	\leq
	\frac{C}{(h-k)^{2^*}}
	|A_{\beta,k}|^{\gamma}, \quad \text{where } \gamma:=2^*
	\left(\frac1{r'}-\frac1{2^*}\right).
	\]
	Observing that
	\[
	2^*
	\left(\frac1{r'}-\frac1{2^*}\right)>1
	\quad\Longleftrightarrow\quad
	\frac{2^*}{r'}>2
	\quad\Longleftrightarrow\quad
	r'<\frac{N}{(N-2)}
	\quad\Longleftrightarrow\quad
	r>\frac N2.
	\]
	and applying Lemma \ref{lemma:stampacchia} to the function $t\mapsto \varphi(t)=|A_{\b,t}|$, we obtain the existence of $D>0$, independent of $i$ and $\b$, such that
	$$
	\left|\{x\in \Omega: w_{i,\b}>D\}\right|=0 \quad \text{for every } i,\b.
	$$
	This implies
	$$
	\sup_{i,\b}\norm{w_{i,\b}}_{L^\infty(\Omega)}\leq D,
	$$
	which is equivalent to
	$$
	\sup_{i,\b}\norm{u_{i,\b}}_{L^\infty(\Omega)}\leq M+D.
	$$
	
	It remains to consider the case $N=2$. Similarly as before, choosing $q>p-1$ and setting
	\[
	r:=\frac{q}{p-1}>1,
	\]
	we obtain
	\[
	-\Delta w_{i,\beta}\leq g_{i,\beta}
	\qquad\text{in the weak sense},
	\]
	with
	\[
	\sup_{i,\beta}\|g_{i,\beta}\|_{L^r(\Omega)}\leq C.
	\]
	Choose now any $t>2r'$. By the Sobolev embedding, the same argument as above gives
	\[
	\|v_{\beta,k}\|_{L^t(\Omega)}
	\leq
	C\,|A_{\beta,k}|^{\frac1{r'}-\frac1t}.
	\]
	Consequently, for every $h>k$,
	\[
	|A_{\beta,h}|
	\leq
	\frac{C}{(h-k)^t}
	|A_{\beta,k}|^\gamma,
	\qquad
	\gamma
	:=
	t\left(\frac1{r'}-\frac1t\right)
	=
	\frac{t}{r'}-1.
	\]
	By the choice $t>2r'$, we have $\gamma>1$ and Lemma \ref{lemma:stampacchia} proves the $L^\infty$ bounds as before. This concludes the proof.
\end{proof} 

\begin{proof}[Conclusion of the proof of Theorem \ref{thm: regolarità al bordo}]
	
	By \eqref{eq:global estimate}, \eqref{eq:growth assumption strong} and uniform $L^\infty$ bounds established in Lemma \ref{lemma:H^1 to Linfty}, we obtain boundedness in $C^{0,\alpha}(\overline{\Omega})$ for every $\alpha\in (0,\bar\nu)$. Then, up to a subsequence, $\mf{u}_\beta \to \tilde{\mf{u}}$ in $C^{0,\alpha}(\overline{\Omega})$, for every $\alpha \in (0,\bar \nu)$, and $\mf{u}_\beta \rightharpoonup \tilde{\mf{u}}$ in $H^1(\Omega)$ by Lemma \ref{lemma:H^1 to Linfty}, for some $\tilde{\mf{u}}\in H^1(\Omega)$ with $\tilde{\mf{u}}-\bs \psi\in H_0^1(\Omega)$.
	
	The $H^1_{\loc}$ convergence and the fact that
	\begin{equation*}
		\beta \int_{K} \sum_{\substack{J\subseteq [d]\\|J|=k}} \gamma_J u_{J,\beta}^2\,dx \to 0 \quad \text{for every compact } K\subset \Omega
	\end{equation*}
	follow from Theorem \ref{thm: regolarità interna}, that clearly can be applied under these stronger assumptions. 
	\end{proof}

\section{Limit problem with partial segregation}\label{sec:limit problem}
In this last section we prove the results concerning the limit problem, namely Theorem \ref{thm:limit problem} and Corollary \ref{cor: regolarità di tutti i minimi}.

\begin{proof}[Proof of Theorem \ref{thm:limit problem}]
	
We recall from Theorem \ref{thm: regolarità al bordo} \ref{thm point 2} that we have already established
\[
\mf u_{\beta_n}\rightharpoonup \tilde{\mf u}
\quad\text{in } H^1(\Omega),
\qquad
\mf u_{\beta_n}\to \tilde{\mf u}
\quad\text{in } C^{0,\alpha}(\overline{\Omega})
\]
for every $\alpha\in(0,\bar\nu)$, with $\tilde{\mf u}-\bs\psi\in H_0^1(\Omega)$ and $\tilde{\mf u}$ satisfying the partial $k$-segregation condition. In other words, $\tilde{\mf u}\in \cH_{\bs\psi}$.

Notice that, by \eqref{eq:convergence pointwise f_n to f} and \eqref{eq:growth assumption strong}, also $\bs f_\infty$ satisfies
\[
|f_{i,\infty}(x,s)|\leq C(1+|s|^{p-1})
\qquad
\text{for a.e. }x\in\Omega,\ \text{for every }s\in[0,+\infty),
\]
with $p<2^*$.

We first claim that
\[
\limsup_{n\to\infty}c_{\b_n,\bs f_{\b_n}}
\leq c_{\infty,\bs f_\infty}.
\]
Indeed, let $\eps>0$. By the definition of $c_{\infty,\bs f_\infty}$, there exists $\mf u_\eps\in \cH_{\bs\psi}$ such that
\begin{equation}\label{eq:8}
	J_{\infty,\bs f_\infty}(\mf u_\eps)
	\leq c_{\infty,\bs f_\infty}+\eps.
\end{equation}
Since $\cH_{\bs\psi}\subset\cU_{\bs\psi}$ and $\mf u_\eps$ satisfies the partial $k$-segregation condition, it is an admissible competitor for $J_{\b_n,\bs f_{\b_n}}$ and its interaction term vanishes identically. Therefore,
\begin{align*}
	c_{\b_n,\bs f_{\b_n}}
	&\leq J_{\b_n,\bs f_{\b_n}}(\mf u_\eps)\\
	&=J_{\infty,\bs f_{\b_n}}(\mf u_\eps)\\
	&=J_{\infty,\bs f_\infty}(\mf u_\eps)
	+\int_\Omega\sum_{i=1}^d
	\left(
	F_{i,\infty}(x,u_{i,\eps})
	-F_{i,\b_n}(x,u_{i,\eps})
	\right)\,dx\\
	&\leq c_{\infty,\bs f_\infty}+\eps+o(1)
	\qquad\text{as }n\to\infty,
\end{align*}
where in the last inequality we used \eqref{eq:8}, together with
\eqref{eq:growth assumption strong}, \eqref{eq:convergence pointwise f_n to f}, and the dominated convergence theorem. Passing to the $\limsup$ and then letting $\eps\to0$, the claim follows.

On the other hand, since $\tilde{\mf u}\in \cH_{\bs\psi}$, by \eqref{eq:growth assumption strong}, \eqref{eq:convergence pointwise f_n to f}, the uniform convergence $\mf u_{\b_n}\to\tilde{\mf u}$, the dominated convergence theorem, and the weak lower semicontinuity of the Dirichlet energy, we obtain
\begin{align*}
	c_{\infty,\bs f_\infty}
	&\leq J_{\infty,\bs f_\infty}(\tilde{\mf u})\\
	&=\sum_{i=1}^d\int_\Omega
	\left(
	\frac12|\nabla\tilde u_i|^2
	-F_{i,\infty}(x,\tilde u_i)
	\right)\,dx\\
	&\leq
	\liminf_{n\to\infty}
	\sum_{i=1}^d\int_\Omega
	\left(
	\frac12|\nabla u_{i,\b_n}|^2
	-F_{i,\b_n}(x,u_{i,\b_n})
	\right)\,dx\\
	&\leq
	\liminf_{n\to\infty}
	\int_\Omega
	\left(
	\sum_{i=1}^d\frac12|\nabla u_{i,\b_n}|^2
	+\frac{\beta_n}{2}
	\sum_{\substack{J\subseteq[d]\\ |J|=k}}
	\gamma_J u_{J,\b_n}^2
	-\sum_{i=1}^dF_{i,\b_n}(x,u_{i,\b_n})
	\right)\,dx\\
	&=
	\liminf_{n\to\infty}
	J_{\b_n,\bs f_{\b_n}}(\mf u_{\b_n})\\
	&=
	\liminf_{n\to\infty}c_{\b_n,\bs f_{\b_n}}\\
	&\leq
	\limsup_{n\to\infty}c_{\b_n,\bs f_{\b_n}}
	\leq c_{\infty,\bs f_\infty}.
\end{align*}

Hence all the inequalities above are equalities. In particular, we obtain the following conclusions:
\begin{itemize}
	\item the Dirichlet energies converge, namely
	\[
	\lim_{n\to\infty}
	\sum_{i=1}^d\int_\Omega|\nabla u_{i,\b_n}|^2\,dx
	=
	\sum_{i=1}^d\int_\Omega|\nabla\tilde u_i|^2\,dx.
	\]
	Since each term is weakly lower semicontinuous, this implies, for every $i=1,\dots,d$,
	\[
	\lim_{n\to\infty}
	\int_\Omega|\nabla u_{i,\b_n}|^2\,dx
	=
	\int_\Omega|\nabla\tilde u_i|^2\,dx.
	\]
	Together with the weak $H^1$ convergence, this yields
	\[
	\mf u_{\b_n}\to\tilde{\mf u}
	\qquad\text{strongly in }H^1(\Omega);
	\]
	
	\item the interaction energy vanishes:
	\[
	\lim_{n\to\infty}
	\beta_n\int_\Omega
	\sum_{\substack{J\subseteq[d]\\ |J|=k}}
	\gamma_Ju_{J,\b_n}^2\,dx=0;
	\]
	
	\item we have
	\[
	J_{\infty,\bs f_\infty}(\tilde{\mf u})
	=c_{\infty,\bs f_\infty},
	\]
	and hence $\tilde{\mf u}$ is a minimizer of the limit problem;
	
	\item the energy levels converge:
	\[
	\lim_{n\to\infty}
	c_{\b_n,\bs f_{\b_n}}
	=c_{\infty,\bs f_\infty}.
	\]
\end{itemize}

This concludes the proof.
\end{proof}

\begin{proof}[Proof of Corollary \ref{cor: regolarità di tutti i minimi}]
	Let $\tilde{\mf u}$ be a minimizer for \eqref{eq:c infty}. We will exhibit a suitable sequence $\{\mf u_{\beta_n}\}_n$ minimizing a family of functionals to which Theorem \ref{thm:limit problem} applies and such that $\mf{u}_{\beta_n}\to\tilde{\mf u}$. This will prove the regularity of $\tilde{\mf u}$.
	
	For every $n\in \N$, let $\beta_n=n$ and consider the following functional, defined for every $\mf u\in \cU_{\bs\psi}$:
	$$
	\tilde J_{n,\bs f}(\mf u)=\frac{1}{2}\int_{\Omega} \Big(\sum_{i=1}^d |\nabla u_i|^2 + n \sum_{\substack{J\subseteq [d]\\|J|=k}} u_J^2+\sum_{i=1}^d \arctan\left((u_i-\tilde u_i)^2\right)-2\sum_{i=1}^d F_{i}(x,u_i(x))\Big)\,dx.
	$$
	We first show that the problem
	\begin{equation} \label{eq:min penalized}
		\inf_{\mathcal{U_{\bs \psi}}}\tilde J_{n,\bs f}.
	\end{equation}
	admits a nonnegative minimizer. 
	
	Notice that it can be written in the form of problem \eqref{eq:def J_beta} as 
	\begin{equation*}
		\tilde J_{n,\bs f}=J_{n,\bs G}, \quad \text{where } G_i(x,u)= F_i(x,u)-\frac{1}{2}\arctan((u-\tilde u_i(x))^2),
	\end{equation*} 
	and the derivative of $G_i$ with respect to the second variable is given by
	$$
	g_i(x,s)=f_i(x,s)-\frac{(s-\tilde u_i(x))}{1+(s-\tilde u_i(x))^4}.
	$$
	Since $\bs f$ satisfies (A1) and (A2) and $\tilde u_i(x)\geq 0$ for a.e. $x\in \Omega$, it is straightforward to check that $\bs g$ also satisfies (A1) and (A2). Therefore, by Proposition \ref{prop:minim at fixed trace exist}, for every $n$ \eqref{eq:min penalized} admits a nonnegative minimizer $\mf u_n$, which is a weak solution to
	\begin{equation}\label{eq:system penalized u_n}
		\begin{cases} 
			\displaystyle -\Delta u_{i,n} = -n u_{i,n} \sum_{\substack{J\subseteq [d]\setminus \{i\}\\ |J|=k-1}} u_{J,n}^2 + f_{i}(x,u_{i,n})-\frac{u_{i,n}-\tilde u_i}{1+(u_{i,n}-\tilde u_i)^4}, \quad u_{i,n}\geq 0 & \text{in }\Omega,\\
			u_{i,n} = \psi_i & \text{on }\partial \Omega.
		\end{cases}
	\end{equation}	
	
	Moreover, since the family of forcing terms obtained setting $\bs g_n:= \bs g$ for every $n$ is independent of $n$, Theorem \ref{thm:limit problem} applies to the sequence $\{\mf u_n\}_n$ (notice that, by Remark \ref{rmk:bounds follow under A1-A2}, $\{\mf u_n\}_n$ is bounded in $H^1(\Omega)$), implying that, up to a subsequence,
	$\mf u_n\to \mf w$ in $H^1(\Omega)$ and in $C^{0,\alpha}(\overline{\Omega})$, for some $\mf w\in \mathcal{H}_{\bs \psi}$ which solves
	\begin{equation}\label{eq:13}
		\tilde J_{\infty,\bs f}(\mf w)=\min_{\mathcal{H}_{\bs \psi}}\tilde J_{\infty,\bs f},
	\end{equation}
	where 
	$$
	\tilde J_{\infty,\bs f}(\mf u)=\frac{1}{2}\int_\Omega\sum_{i=1}^{d} \left( |\nabla u_i|^2+\arctan\left((u_i-\tilde u_i)^2\right)-2F_i(x,u_i(x))\right)\, dx.
	$$
	Theorem \ref{thm:limit problem} also proves that
	\begin{equation}\label{eq:interazione arg penalizz va a zero}
		n \int_{\Omega} \sum_{\substack{J\subseteq [d]\\|J|=k}} u_{J,n}^2\,dx \to 0 \quad \text{and} \quad \min_{\mathcal{U_{\bs \psi}}} \tilde J_{n,\bs f}\to \min_{\mathcal{H}_{\bs \psi}}\tilde J_{\infty,\bs f} \quad \text{as } n\to\infty.
	\end{equation}
	
	We claim that $\mf w=\tilde{\mf{u}}$.
	Clearly, since $\tilde{\mf u}$ minimizes \eqref{eq:c infty} while $\mf w$ minimizes \eqref{eq:13} and both belong to $\mathcal{H}_{\bs \psi}$,
	\begin{align*}
		J_{\infty,\bs f}(\mf w)&\geq
		J_{\infty,\bs f}(\tilde{\mf u})
		= \tilde J_{\infty,\bs f}(\tilde{\mf u})
		\geq \tilde J_{\infty,\bs f}(\mf w)
		= \tilde J_{n,\bs f}(\mf u_n)+o(1)\\
		&= \frac{1}{2}\int_\Omega \sum_{i=1}^{d}
		\left(|\nabla u_{i,n}|^2+\arctan\left((u_{i,n}-\tilde u_i)^2\right)
		-2F_i(x,u_{i,n}(x))\right)\, dx + o(1)\\
		&= \frac{1}{2}\int_\Omega \sum_{i=1}^{d}
		\left(|\nabla w_i|^2-2F_i(x,w_i(x))\right)\, dx
		+\frac{1}{2}\int_\Omega \sum_{i=1}^{d}
		\arctan\left((u_{i,n}-\tilde u_i)^2\right) \, dx + o(1)\\
		&=J_{\infty,\bs f}(\mf w)+\frac{1}{2}\int_\Omega \sum_{i=1}^{d}
		\arctan\left((u_{i,n}-\tilde u_i)^2\right) \, dx + o(1).
	\end{align*}
	as $n\to\infty$, thanks to $H^1$ and uniform convergence and \eqref{eq:interazione arg penalizz va a zero}.
	This yields
	$$
	\lim_{n\to\infty} \int_\Omega \arctan\left((u_{i,n}-\tilde u_i)^2\right)\, dx = 0 \quad \text{for every $i=1,\dots,d.$}
	$$
	
	This clearly implies that $u_{i,n}\to \tilde u_i$ almost everywhere in $\Omega$ as $n\to \infty$, and this proves the claim recalling that $\mf u_{n}\to \mf w$ uniformly.
	
	To prove \eqref{eq:equaz in insieme positività}, fix $x_0\in \Omega$ such that $\tilde u_i(x_0)>0$. Then, by continuity and uniform convergence, $\inf_{B_\rho(x_0)} u_{i,n} \geq \delta$ for some $\rho,\delta>0$. Observe that, for any $\varphi\in C^\infty_c(B_\rho(x_0))$ and $J\subset [d]$, $|J|=k$,
	$$
	\left|\int_{B_\rho(x_0)} n \varphi u_{i,n}\sum_{\substack{J\subseteq [d]\setminus \{i\}\\ |J|=k-1}} u_{J,n}^2\right| \leq \frac{1}{\delta}\int_{B_\rho(x_0)} n |\varphi| u_{i,n}^2\sum_{\substack{J\subseteq [d]\setminus \{i\}\\ |J|=k-1}} u_{J,n}^2 \to 0 \quad \text{as } n\to\infty
	$$
	by \eqref{eq:interazione arg penalizz va a zero}.
	Then, passing to the limit in \eqref{eq:system penalized u_n} and using that $\mf u_n\to \tilde{\mf u}$ in $H^1$ we obtain that
	$$
	-\Delta \tilde u_i = f_i(x,\tilde u_i(x)) \quad \text{in } B_\rho(x_0).
	$$
	
	The proof of \eqref{eq:pohozaev} is very similar to the one in \cite[Sec. 6]{SoTe1}, thus we only give a sketch here: first of all, multiplying each equation in \eqref{eq:system penalized u_n} by $\nabla u_{i,n}(x)\cdot(x-x_0)$, integrating over $B_r(x_0)\subset \Omega$ and summing over $i$, we obtain after some rearrangements as in the classical Pohozaev identity
	\begin{multline*}
		r\int_{S_r(x_0)} \Big(\sum_{i=1}^{d}|\nabla u_{i,n}|^2+n\sum_{\substack{J\subseteq [d] \\ |J|=k}}u_{J,n}^2\Big)\, d\sigma = (N-2)\int_{B_r(x_0)}\sum_{i=1}^{d} |\nabla u_{i,n}|^2+Nn\int_{B_r(x_0)}\sum_{\substack{J\subseteq [d] \\ |J|=k}} u_{J,n}^2 \\
		+ 2r\int_{S_r(x_0)} \sum_{i=1}^{d}(\partial_\nu u_{i,n})^2 + 2r \sum_{i=1}^{d} \int_{S_r(x_0)}\left(F_i(x,u_{i,n})-\frac{1}{2}\arctan\left((u_{i,n}-\tilde u_i)^2\right)\right)\\
		-2\sum_{i=1}^{d}\Big[N\int_{B_r(x_0)}\left(F_i(x,u_{i,n})-\frac{1}{2}\arctan\left((u_{i,n}-\tilde u_i)^2\right)\right)+\int_{B_r(x_0)}\nabla_x F_i(x,u_{i,n})\cdot (x-x_0) \Big] \\
		+2N\sum_{i=1}^{d} \int_{B_r(x_0)} \frac{u_{i,n}-\tilde u_i}{1+(u_{i,n}-\tilde u_i)^4}\nabla \tilde u_{i}(x)\cdot(x-x_0).
	\end{multline*}
	Let now $x_0\in\Omega$ and $0<r<\dist(x_0,\partial\Omega)$. Integrating the last equation for $s\in (\bar r,r)$ for some $\bar r<r$ we obtain
	\begin{align*}
		\int_{B_r(x_0)\setminus B_{\bar r}(x_0)} &\Big(\sum_{i=1}^{d}|\nabla u_{i,n}|^2+n\sum_{\substack{J\subseteq [d] \\ |J|=k}}u_{J,n}^2\Big) = \int_{\bar r}^r \frac{(N-2)}{s}\Big(\int_{B_s(x_0)}\sum_{i=1}^{d} |\nabla u_{i,n}|^2 \,\Big) ds\\
		&+\int_{\bar r}^r \frac{N}{s}n\Big(\int_{B_s(x_0)}\sum_{\substack{J\subseteq [d] \\ |J|=k}} u_{J,n}^2\Big)\, ds + 2\int_{B_r(x_0)\setminus B_{\bar r}(x_0)} \sum_{i=1}^{d}\Big(\nabla u_{i,n}(x)\cdot \frac{x-x_0}{|x-x_0|}\Big)^2 \\
		&+ 2 \sum_{i=1}^{d} \int_{B_r(x_0)\setminus B_{\bar r}(x_0)}\left(F_i(x,u_{i,n})-\arctan\left((u_{i,n}-\tilde u_i)^2\right)\right)	\\
		&-2\sum_{i=1}^{d}\int_{\bar r}^r\frac{N}{s}\Big(\int_{B_s(x_0)}\left(F_i(x,u_{i,n})-\arctan\left((u_{i,n}-\tilde u_i)^2\right)\right)\,\Big) ds\\
		&-2\sum_{i=1}^{d}\int_{\bar r}^r \frac{1}{s}\Big(\int_{B_s(x_0)}\nabla_x F_i(x,u_{i,n})\cdot (x-x_0)\Big)\, ds \\
		&+2\sum_{i=1}^{d}\int_{\bar r}^r \frac{1}{s}\Big(\int_{B_s(x_0)}\frac{u_{i,n}-\tilde u_i}{1+(u_{i,n}-\tilde u_i)^4}\nabla \tilde u_{i}(x)\cdot (x-x_0)\Big)\, ds.
	\end{align*}
	Observe that, by the a.e. convergence $u_{i,n}\to\tilde u_i$ and the
	assumptions on $\nabla_xF_i$, the dominated convergence theorem yields
	\[
	\nabla_xF_i(x,u_{i,n})
	\to
	\nabla_xF_i(x,\tilde u_i)
	\qquad\text{in }L^1(B_r(x_0)).
	\] 
	Then, recalling also \eqref{eq:interazione arg penalizz va a zero} and $H^1$ convergence, letting $n\to \infty$ we obtain
	\begin{align*}
		\int_{B_r(x_0)\setminus B_{\bar r}(x_0)} &\sum_{i=1}^{d}|\nabla \tilde u_{i}|^2 = \int_{\bar r}^r \frac{(N-2)}{s}\Big(\int_{B_s(x_0)}\sum_{i=1}^{d} |\nabla \tilde u_{i}|^2 \,\Big) ds\\
		& + 2\int_{B_r(x_0)\setminus B_{\bar r}(x_0)} \sum_{i=1}^{d}\Big(\nabla \tilde u_{i}(x)\cdot \frac{x-x_0}{|x-x_0|}\Big)^2 \\
		&+ 2 \sum_{i=1}^{d} \int_{B_r(x_0)\setminus B_{\bar r}(x_0)}F_i(x,\tilde u_{i})	\\
		&-2\sum_{i=1}^{d}\int_{\bar r}^r\left(\frac{N}{s}\Big(\int_{B_s(x_0)}F_i(x,\tilde u_{i})\Big) + \frac{1}{s}\Big(\int_{B_s(x_0)}\nabla_x F_i(x,\tilde u_i)\cdot (x-x_0)\Big)\right)\, ds.
	\end{align*}
	Differentiating with respect to $r$ the above equation, we conclude.
\end{proof}

\subsection*{Acknowledgements}
The author was partially supported by the INDAM--GNAMPA 2026 research project ``Ottimizzazione Geometrica e Frontiere Libere''.
The author would also like to thank Nicola Soave for the fruitful discussions during the preparation of this work.
Finally, the author is grateful to the anonymous referees for their careful reading of the manuscript and for several insightful comments and suggestions, which contributed to improving the original version of the paper.

\printbibliography[heading=mybib]

\end{document}